\documentclass[11pt, a4paper, twoside]{amsart}
\usepackage[centering, totalwidth = 370pt, totalheight = 610pt]{geometry}
\usepackage{amssymb, amsmath, amsthm, enumerate, microtype, stmaryrd, url, mathpartir}
\usepackage[latin1]{inputenc}
\usepackage[dvips, arrow, matrix, tips, curve]{xy}
\usepackage{color}
\definecolor{darkgreen}{rgb}{0,0.45,0}
\usepackage[pagebackref,colorlinks,citecolor=darkgreen,linkcolor=darkgreen]{hyperref}
\SelectTips{cm}{10}

\DeclareMathOperator{\Id}{Id} \DeclareMathOperator{\ob}{ob}

 \DeclareMathOperator{\mor}{mor}
\newcommand{\cat}[1]{\mathrm{#1}}

\newcommand{\id}{\mathrm{id}}
\newcommand{\thg}{{\mathord{\text{--}}}}

\newcommand{\set}[2]{\left\{\,#1 \ \vrule\  #2\,\right\}}

\newcommand{\defeq}{\mathrel{\mathop:}=}

\newcommand{\cd}[2][]{\vcenter{\hbox{\xymatrix#1{#2}}}}

\newcommand{\C}{{\mathcal C}}

\newcommand{\E}{{\mathcal E}}

\newcommand{\I}{{\mathcal I}}
\newcommand{\J}{{\mathcal J}}

\renewcommand{\P}{{\mathcal P}}

\newcommand{\V}{{\mathcal V}}
\newcommand{\W}{{\mathcal W}}

\newcommand{\xtor}[1]{\cdl[@1]{{} \ar[r]|-{\object@{|}}^{#1} & {}}}

\makeatletter

\def\hookleftarrowfill@{\arrowfill@\leftarrow\relbar{\relbar\joinrel\rhook}}
\def\twoheadleftarrowfill@{\arrowfill@\twoheadleftarrow\relbar\relbar}
\def\leftbararrowfill@{\arrowdoublefill@{\leftarrow\mkern-5mu}\relbar\mapstochar\relbar\relbar}
\def\Leftbararrowfill@{\arrowdoublefill@{\Leftarrow\mkern-2mu}\Relbar\Mapstochar\Relbar\Relbar}
\def\leftringarrowfill@{\arrowdoublefill@{\leftarrow\mkern-3mu}\relbar{\mkern-3mu\circ\mkern-2mu}\relbar\relbar}
\def\lefttriarrowfill@{\arrowfill@{\mathrel\triangleleft\mkern0.5mu\joinrel\relbar}\relbar\relbar}
\def\Lefttriarrowfill@{\arrowfill@{\mathrel\triangleleft\mkern1mu\joinrel\Relbar}\Relbar\Relbar}

\def\hookrightarrowfill@{\arrowfill@{\lhook\joinrel\relbar}\relbar\rightarrow}
\def\twoheadrightarrowfill@{\arrowfill@\relbar\relbar\twoheadrightarrow}
\def\rightbararrowfill@{\arrowdoublefill@{\relbar\mkern-0.5mu}\relbar\mapstochar\relbar\rightarrow}
\def\Rightbararrowfill@{\arrowdoublefill@{\Relbar\mkern-2mu}\Relbar\Mapstochar\Relbar\Rightarrow}
\def\rightringarrowfill@{\arrowdoublefill@\relbar\relbar{\mkern-2mu\circ\mkern-3mu}\relbar{\mkern-3mu\rightarrow}}
\def\righttriarrowfill@{\arrowfill@\relbar\relbar{\relbar\joinrel\mkern0.5mu\mathrel\triangleright}}
\def\Righttriarrowfill@{\arrowfill@\Relbar\Relbar{\Relbar\joinrel\mkern1mu\mathrel\triangleright}}

\def\leftrightarrowfill@{\arrowfill@\leftarrow\relbar\rightarrow}
\def\mapstofill@{\arrowfill@{\mapstochar\relbar}\relbar\rightarrow}

\renewcommand*\xleftarrow[2][]{\ext@arrow 20{20}0\leftarrowfill@{#1}{#2}}
\providecommand*\xLeftarrow[2][]{\ext@arrow 60{22}0{\Leftarrowfill@}{#1}{#2}}
\providecommand*\xhookleftarrow[2][]{\ext@arrow 10{20}0\hookleftarrowfill@{#1}{#2}}
\providecommand*\xtwoheadleftarrow[2][]{\ext@arrow 60{20}0\twoheadleftarrowfill@{#1}{#2}}
\providecommand*\xleftbararrow[2][]{\ext@arrow 10{22}0\leftbararrowfill@{#1}{#2}}
\providecommand*\xLeftbararrow[2][]{\ext@arrow 50{24}0\Leftbararrowfill@{#1}{#2}}
\providecommand*\xleftringarrow[2][]{\ext@arrow 10{26}0\leftringarrowfill@{#1}{#2}}
\providecommand*\xlefttriarrow[2][]{\ext@arrow 80{24}0\lefttriarrowfill@{#1}{#2}}
\providecommand*\xLefttriarrow[2][]{\ext@arrow 80{24}0\Lefttriarrowfill@{#1}{#2}}

\renewcommand*\xrightarrow[2][]{\ext@arrow 01{20}0\rightarrowfill@{#1}{#2}}
\providecommand*\xRightarrow[2][]{\ext@arrow 04{22}0{\Rightarrowfill@}{#1}{#2}}
\providecommand*\xhookrightarrow[2][]{\ext@arrow 00{20}0\hookrightarrowfill@{#1}{#2}}
\providecommand*\xtwoheadrightarrow[2][]{\ext@arrow 03{20}0\twoheadrightarrowfill@{#1}{#2}}
\providecommand*\xrightbararrow[2][]{\ext@arrow 01{22}0\rightbararrowfill@{#1}{#2}}
\providecommand*\xRightbararrow[2][]{\ext@arrow 04{24}0\Rightbararrowfill@{#1}{#2}}
\providecommand*\xrightringarrow[2][]{\ext@arrow 01{26}0\rightringarrowfill@{#1}{#2}}
\providecommand*\xrighttriarrow[2][]{\ext@arrow 07{24}0\righttriarrowfill@{#1}{#2}}
\providecommand*\xRighttriarrow[2][]{\ext@arrow 07{24}0\Righttriarrowfill@{#1}{#2}}

\providecommand*\xmapsto[2][]{\ext@arrow 01{20}0\mapstofill@{#1}{#2}}
\providecommand*\xleftrightarrow[2][]{\ext@arrow 10{22}0\leftrightarrowfill@{#1}{#2}}
\providecommand*\xLeftrightarrow[2][]{\ext@arrow 10{27}0{\Leftrightarrowfill@}{#1}{#2}}

\makeatother

\newcommand{\twocong}[2][0.5]{\ar@{}[#2] \save ?(#1)*{\cong}\restore}
\newcommand{\twoeq}[2][0.5]{\ar@{}[#2] \save ?(#1)*{=}\restore}
\newcommand{\rtwocell}[3][0.5]{\ar@{}[#2] \ar@{=>}?(#1)+/l 0.2cm/;?(#1)+/r 0.2cm/^{#3}}
\newcommand{\ltwocell}[3][0.5]{\ar@{}[#2] \ar@{=>}?(#1)+/r 0.2cm/;?(#1)+/l 0.2cm/^{#3}}
\newcommand{\ltwocello}[3][0.5]{\ar@{}[#2] \ar@{=>}?(#1)+/r 0.2cm/;?(#1)+/l 0.2cm/_{#3}}
\newcommand{\dtwocell}[3][0.5]{\ar@{}[#2] \ar@{=>}?(#1)+/u  0.2cm/;?(#1)+/d 0.2cm/^{#3}}
\newcommand{\dltwocell}[3][0.5]{\ar@{}[#2] \ar@{=>}?(#1)+/ur  0.2cm/;?(#1)+/dl 0.2cm/^{#3}}
\newcommand{\drtwocell}[3][0.5]{\ar@{}[#2] \ar@{=>}?(#1)+/ul  0.2cm/;?(#1)+/dr 0.2cm/^{#3}}
\newcommand{\dthreecell}[3][0.5]{\ar@{}[#2] \ar@3{->}?(#1)+/u  0.2cm/;?(#1)+/d 0.2cm/^{#3}}
\newcommand{\utwocell}[3][0.5]{\ar@{}[#2] \ar@{=>}?(#1)+/d 0.2cm/;?(#1)+/u 0.2cm/_{#3}}
\newcommand{\dtwocelltarg}[3][0.5]{\ar@{}#2 \ar@{=>}?(#1)+/u  0.2cm/;?(#1)+/d 0.2cm/^{#3}}
\newcommand{\utwocelltarg}[3][0.5]{\ar@{}#2 \ar@{=>}?(#1)+/d  0.2cm/;?(#1)+/u 0.2cm/_{#3}}

\newdir{(}{{}*!<0em,-.14em>-\cir<.14em>{l^r}}
\newdir{ (}{{}*!/-5pt/\dir{(}}
\newdir{ >}{{}*!/-5pt/\dir{>}}

\swapnumbers
\theoremstyle{plain}
\newtheorem{Thm}{Theorem}[subsection]
\newtheorem{Prop}[Thm]{Proposition}

\newtheorem{Lemma}[Thm]{Lemma}

\theoremstyle{definition}
\newtheorem{Defn}[Thm]{Definition}

\theoremstyle{remark}

\newtheorem{Rk}[Thm]{Remark}

\newcommand{\ty}{\mathsf{type}}
\renewcommand{\c}{,\,\,}

\renewcommand{\r}{\mathrm{r}}
\renewcommand{\J}{\mathrm{J}}
\newcommand{\ctxt}{\mathsf{ctxt}}

\begin{document}
 \leftmargini=2em
\title{Types are weak $\omega$-groupoids}
\author{Benno van den Berg}
\address{Technische Universit\"at Darmstadt, Fachbereich Mathematik, Schlo\ss garten\-stra\ss e 7, 64289 Darmstadt, Germany}
\email{berg@mathematik.tu-darmstadt.de}
\author{Richard Garner}
\address{Department of Pure Mathematics and Mathematical Statistics, University of Cambridge, Cambridge CB3 0WB, UK}
\email{rhgg2@cam.ac.uk} 
\begin{abstract}
We define a notion of weak $\omega$-category internal to a model of
Martin-L\"of type theory, and prove that each type bears a canonical weak
$\omega$-category structure obtained from the tower of iterated identity types
over that type. We show that the $\omega$-categories arising in this way are in
fact $\omega$-groupoids.
\end{abstract}
 \maketitle

\renewcommand{\vec}[1]{#1}
\section{Introduction}
It has long been understood that there is a close connection between algebraic
topology and higher-dimensional category
theory~\cite{Grothendieck1983Pursuing}. More recently, it has become apparent
that both are in turn related to the intensional type theory of
Martin-L\"of~\cite{Nordstrom1990Programming}. Whilst attempts to make this link
precise have only borne fruit in the past few
years~\cite{Awodey2008Homotopy,Gambino2008identity,Garner2008Two-dimensional,Warren2008Homotopy},
the basic idea dates back to an observation made by Hofmann and Streicher
in~\cite{Hofmann1998groupoid}. Recall that in Martin-L\"of type theory, we may
construct from a type $A$ and elements $a, b \in A$, a new type $\Id(a,b)$
whose elements are to be thought of as proofs that $a$ and $b$ are
propositionally equal. Hofmann and Streicher observe that the type-theoretic
functions
\begin{align*}
    1 &\to \Id(a, a)\,\text, &
    \Id(b, c) \times \Id(a, b) & \to \Id(a, c) &
    \text{and} \quad \Id(a, b) & \to \Id(b, a)
\end{align*}
expressing the reflexivity, transitivity and symmetry of propositional equality
allow us to view the type $A$ as a \emph{groupoid}---a category whose every
morphism is invertible---wherein objects are elements $a \in A$ and morphisms
$a \to b$ are elements $p \in \Id(a, b)$. However, as is made clear
in~\cite{Hofmann1998groupoid}, this is not the end of the story. The groupoid
axioms for $A$ hold only ``up to propositional equality''; which is to say
that, for example, the associativity diagram
\begin{equation*}
   \cd{
    \Id(c,d) \times \Id(b,c) \times \Id(a,b) \ar[d] \ar[r] &
    \Id(c,d) \times \Id(a,c) \ar[d] \\
    \Id(b,d) \times \Id(a,b) \ar[r] &
    \Id(a,d)
   }
\end{equation*}
does not commute on the nose, but only up to suitable terms
\begin{equation*}
    \alpha_{p,q,r} \in \Id\big( r \circ (q \circ p), \, (r \circ q) \circ
    p\big)\qquad \text{($p \in \Id(a,b)$, $q \in \Id(b,c)$, $r \in
    \Id(c,d)$).}
\end{equation*}
Thus, if we wish to view $A$ as an honest groupoid, we must first quotient out
the sets of elements $p \in \Id(a,b)$ by propositional equality. A more
familiar instance of the same phenomenon occurs in constructing the fundamental
groupoid of a space---where we must identify paths up to homotopy---and this
suggests the following analogy: that types are like topological spaces, and
that propositional equality is like the homotopy relation. Using the machinery
of abstract homotopy theory, this analogy has been given a precise form
in~\cite{Awodey2008Homotopy}, which constructs type theoretic structures from
homotopy theoretic ones, and in~\cite{Gambino2008identity}, which does the
converse.

\looseness=-1 The connection with algebraic topology in turn suggests the one
with higher-dimensional category theory. A more sophisticated construction of
the fundamental groupoid of a space (suggested
in~\cite{Grothendieck1983Pursuing} and made rigorous
in~\cite{Batanin1998Monoidal}) does not quotient out paths by the homotopy
relation; but instead incorporates these homotopies---and all higher homotopies
between them---into an infinite-dimensional categorical structure known as a
\emph{weak $\omega$-groupoid}, whose various identities, compositions and
inverses satisfy coherence laws, not strictly, but ``up to all higher
homotopies''. This leads us to ask whether the construction of the
type-theoretic ``fundamental groupoid'' admits a similar refinement, which
constructs a weak $\omega$-groupoid from a type by considering not just
elements of the type, and proofs of their equality, but also proofs of equality
between such proofs, and so on. The principal aim of this paper is to show this
to be the case.

In order to give the proof, we must first choose an appropriate notion of weak
$\omega$-groupoid to work with; and since, in the literature, weak
$\omega$-groupoids are studied in the broader context of \emph{weak
$\omega$-categories}---which are ``weak $\omega$-groupoids without the
inverses''---this is tantamount to choosing an appropriate notion of weak
$\omega$-category. There are a number of definitions to pick from, and these
differ from each other both in their general approach and in the details;
see~\cite{Leinster2002survey} for an overview. Of these, it is the definition
of Batanin~\cite{Batanin1998Monoidal} which matches the type theory most
closely, for the following two reasons. Firstly, its basic cellular data are
\emph{globular}: which is to say that an $n$-cell $\alpha \colon x \to y$ can
only exist between a pair of parallel $(n-1)$-cells $x, y \colon f\to g$. A
corresponding property holds for proofs of equality in type theory: to know
that $\alpha \in \Id(x,y)$, we must first know that $x$ and $y$ inhabit the
same type $\Id(f,g)$. Secondly, Batanin's definition is \emph{algebraic}: which
is to say that composition operations are explicitly specified, rather than
merely asserted to exist. This accords with the constructivist notion---central
to the spirit of intensional type theory---that to know something to exist is
nothing less than to be provided with a witness to that fact. On these grounds,
it is Batanin's definition which we will adopt here; or rather, a mild
reformulation of his definition given by Leinster in~\cite{Leinster2002survey}.

The paper is arranged as follows. In Section~\ref{prep}, we recall Batanin's
theory of weak $\omega$-categories, the appropriate specialisation to weak
$\omega$-groupoids, and the necessary background from intensional type theory.
Then in Section~\ref{sec3} we give the proof of our main result. We begin in
Section~\ref{overview} with an explicitly type-theoretic, but informal,
account. When we come to make this precise, it turns out to be convenient to
isolate just those categorical properties of the type theory which make the
proof go through, and then to work in an axiomatic setting assuming only these.
We describe this setting in Section~\ref{axiomatic}, and then in
Sections~\ref{proof1} and~\ref{proof2}, use it to give a formal proof that
every type is a weak $\omega$-groupoid.

It seems appropriate to say a few words about the history of this paper. The
main result was described by the first-named author in 2006 in a presentation
at the workshop ``Identity Types---Topological and Categorical Structure'' held
at Uppsala University~\cite{BergTypes}. The details of the proof were then
worked out by both authors during a 2008 visit by the first author to Uppsala;
and it was at this stage that the axiomatic approach was introduced. Whilst
preparing this manuscript for publication, we become aware that, independently,
Peter Lumsdaine had been considering the same question. His analysis may be
found in~\cite{LumsdaineInpressWeak}. Let us remark only that, where our
argument is category-theoretic in nature, that given by Lumsdaine is
essentially proof-theoretic. We gratefully acknowledge the support of Uppsala
University's Department of Mathematics, and extend our thanks to Erik Palmgren
for organising the aforementioned workshop. The second-named author also
acknowledges the support of a Research Fellowship of St John's College,
Cambridge and a Marie Curie Intra-European Fellowship, Project No.\ 040802.

\section{Preparatory material}\label{prep}
In this section, we review the material necessary for our main result; firstly,
from higher category theory, and secondly, from Martin-L\"of type theory.

\subsection{Weak $\omega$-categories and weak $\omega$-groupoids}
As mentioned in the Introduction, the most appropriate definition of weak
$\omega$-category for our purposes is that of~\cite{Batanin1998Monoidal}, which
describes them as globular sets equipped with algebraic structure. A
\emph{globular set} is a diagram of sets and functions
\begin{equation*}
    \cd[@C-0.5em]{
    X_0 & X_1 \ar@<3pt>[l]^-{t} \ar@<-3pt>[l]_-{s} & X_2 \ar@<3pt>[l]^-{t} \ar@<-3pt>[l]_-{s} & X_3\ar@<3pt>[l]^-{t} \ar@<-3pt>[l]_-{s} & \,\cdots \ar@<3pt>[l]^-{t} \ar@<-3pt>[l]_-{s}}
\end{equation*}
satisfying the globularity equations $ss = st$ and $ts = tt$. We refer to
elements $x \in X_n$ as \emph{$n$-cells} of $X$, and write them as $x \colon sx
\to tx$. In this terminology, the globularity equations express that any
$(n+2)$-cell $f \to g$ must mediate between $(n+1)$-cells $f$ and $g$ which are
\emph{parallel}, in the sense of having the same source and target. Globular
sets also have a coinductive characterisation: to give a globular set $X$ is to
give a set $\ob X$ of objects, and for each $x, y \in \ob X$, a globular set
$X(x,y)$.

The algebraic structure required to make a globular set into a weak
\mbox{$\omega$-category} is encoded by any one of a certain class of monads on
the category of globular sets: those arising from \emph{normalised,
contractible, globular operads}. Informally, such monads are obtained by
``deforming'' the monad $T$ whose algebras are strict $\omega$-categories. To
make this precise, we must first recall some details concerning strict
$\omega$-categories.

If $\V$ is any category with finite products, then one can speak of categories
enriched in $\V$, and of $\V$-enriched functors between
them~\cite{Kelly1982Basic}. The category $\V\text-\cat{Cat}$ of small
$\V$-categories is then itself a category with finite products, so that we can
iterate the process; and when we do so starting from $\V = 1$, we obtain the
sequence $1$,~$\cat{Set}$, $\cat{Cat}$, $2\text-\cat{Cat}$, \dots, whose $n$th
term is the category of small strict $(n-1)$-categories. Now, because any
finite-product preserving functor \mbox{$\V \to \W$} induces a finite-product
preserving functor $\V\text-\cat{Cat} \to \W\text-\cat{Cat}$, we obtain, by
iteration on the unique functor $\cat{Set} \to 1$, a chain
\begin{equation*}
    \cdots \to n\text-\cat{Cat} \cdots \to 2\text-\cat{Cat} \to \cat{Cat} \to
    \cat{Set} \to 1\ \text;
\end{equation*}
and $\omega$-$\cat{Cat}$, the category of small strict $\omega$-categories, is
the limit of this sequence. Unfolding this definition, we find that a strict
$\omega$-category is given by first, an underlying globular set; next,
operations of identity and composition: so for each $n$-cell $x$, an
$(n+1)$-cell $\id_x \colon x \to x$, and for each pair of $n$-cells $f$ and $g$
sharing a $k$-cell boundary (for $k < n$), a composite $n$-cell $g \circ_k f$;
and finally, axioms which express that any two ways of composing a diagram of
$n$-cells using the above operations yield the same result.

There is an evident forgetful functor $U \colon \omega\text-\cat{Cat} \to
\cat{GSet}$, where $\cat{GSet}$ denotes the category of globular sets; and it
is shown in~\cite[Appendix B]{Leinster2004Operads} that this has a left adjoint
and is finitarily monadic. The corresponding monad $T$ on the category of
globular sets may be described as follows. First we give an inductive
characterisation of $T1$, its value at the globular set with one cell in every
dimension. We have:
\begin{itemize}
\item $(T1)_0 = \{\star\}$; and
\item $(T1)_{n+1} = \set{(\pi_1, \dots, \pi_k)}{k \in \mathbb N, \pi_1,
    \dots, \pi_k \in (T1)_n}$.
\end{itemize}
The source and target maps $s, t \colon (T1)_{n+1} \to (T1)_n$ coincide and we
follow~\cite{Leinster2004Operads} in writing $\partial$ for the common value.
This too may be described inductively:
\begin{itemize}
\item $\partial(\pi) = \star$ for $\pi \in (T1)_1$; and
\item $\partial(\pi_1, \dots, \pi_k) = (\partial(\pi_1), \dots,
    \partial(\pi_k))$ otherwise.
\end{itemize}
We regard elements of $(T1)_n$ as indexing possible shapes for pasting diagrams
of $n$-cells. For example, $((\ast), (\ast, \ast)) \in (T1)_2$ corresponds to
the shape

\begin{equation}\label{shape}
    \cd{
      \bullet \ar@/^1em/[r] \ar@/_1em/[r] \dtwocell{r}{} &
      \bullet \ar@/^1.5em/[r]|{}="a" \ar[r]|{}="b" \ar@/_1.5em/[r]|{}="c" &
      \bullet \ar@{=>}"a"+/d  0.15cm/;"b"+/u 0.15cm/ \ar@{=>}"b"+/d  0.15cm/;"c"+/u 0.15cm/
    }\ \text.
\end{equation}
\vskip\baselineskip \noindent We can make this formal as follows. By induction,
we associate to each element $\pi \in (T1)_n$ a globular set $\hat \pi$ which
is the ``shape indexed by $\pi$'':
\begin{itemize}
\item If $\pi = \star$, then $\hat \pi$ is the globular set with $\ob \hat
    \pi = \{\bullet\}$ and $\hat \pi(\bullet, \bullet) = \emptyset$.
\item If $\pi = (\pi_1, \dots, \pi_k)$, then $\hat \pi$ is the globular set
    with $\ob \hat \pi = \{0, \dots, k\}$, $\hat \pi(i-1, i) =
    \widehat{\pi_i}$ (for $1 \leqslant i \leqslant k$), and $\hat \pi(i, j)
    = \emptyset$ otherwise.
\end{itemize}
By a further induction, we define source and target embeddings $\sigma, \tau
\colon \widehat{\partial \pi} \to \hat \pi$:
\begin{itemize}
\item For $\pi \in (T1)_1$, the maps $\sigma, \tau \colon \hat \star \to
    \hat \pi$ send the unique object of $\hat \star$ to the smallest and
    largest elements of $\ob \hat \pi$, respectively.
\item Otherwise, for $\pi = (\pi_1, \dots, \pi_k)$ the morphisms $\sigma$
    and $\tau$ are the identity on objects and map $\widehat{\partial
    \pi}(i-1, i)$ into $ \widehat \pi(i-1, i)$ via $\sigma, \tau \colon
    \widehat{\partial \pi_i} \to \widehat{\pi_i}$.
\end{itemize}
Taken together, these data---the globular set $T1$, the globular sets $\hat
\pi$ and the maps $\sigma$ and $\tau$---completely determine the functor $T$;
this by virtue of it being \emph{familially representable} in the sense
of~\cite[Definition C.3.1]{Leinster2004Higher} (though see
also~\cite{Carboni1995Connected}). Explicitly, $TX$ is the globular set whose
cells are pasting diagrams labelled with cells of $X$:
\begin{equation*}
    (TX)_n = \sum_{\pi \in (T1)_n} \cat{GSet}(\hat \pi, X)\ \text,
\end{equation*}
and whose source and target maps are induced in an obvious way by the maps
$\sigma$ and $\tau$. The unit and multiplication of the monad $T$ are
\emph{cartesian} natural transformations---which is to say that all of their
naturality squares are pullbacks---from which it follows that these are in turn
determined by the components $\eta_1 \colon 1 \to T1$ and $\mu_1 \colon TT1 \to
T1$. The former map associates to the unique $n$-cell of $1$ the pasting
diagram $\iota_n
\defeq (\cdots (\star)\cdots) \in (T1)_n$, whilst the latter sends a typical
element
\begin{equation*}
    (\pi \in (T1)_n\c \phi \colon \hat \pi \to T1)
\end{equation*}
of $(TT1)_n$ to the element $\phi \circ \pi \in (T1)_n$ obtained by
substituting into $\pi$ the pasting diagrams which $\phi$ indexes
(see~\cite[Section 4.2]{Leinster2004Operads} for a pictorial account of this
process).

%
A \emph{globular operad} can now be defined rather succinctly: it is a monad
$P$ on $\cat{GSet}$ equipped with a cartesian monad morphism $\rho \colon P
\Rightarrow T$. The cartesianness of $\rho$ implies that the functor part of
$P$ is determined by its component at $1$ together with the augmentation map
$\rho_1 \colon P1 \to T1$, and it will be convenient to have a description of
$P$ in these terms. Given $\pi \in (T1)_n$, we write $P_\pi$ for the set of
those $\theta \in (P1)_n$ which are mapped to $\pi$ by $\rho_1$, and write $s,
t \colon P_\pi \to P_{\partial \pi}$ for the corresponding restriction of the
source and target maps of $P1$. The value of $P$ at an arbitrary globular set
$X$ is now given (up to isomorphism) by
\begin{equation}\label{Pdesc}
    (PX)_n = \sum_{\pi \in (T1)_n} P_\pi \times \cat{GSet}(\hat \pi,
    X)\ \text,
\end{equation}
with the source and target maps determined in the obvious way. Thus, if we
think of a $T$-algebra structure on $X$ as providing a unique way of composing
each $X$-labelled pasting diagram of shape $\pi$, then a $P$-algebra structure
provides a set of possible ways of composing such diagrams, indexed by the
elements of $P_\pi$.

It follows from the cartesianness of $\rho$ that the unit and the
multiplication of $P$ are themselves cartesian natural transformations, and
hence determined by their components $\eta_1 \colon 1 \to P1$ and $\mu_1 \colon
PP1 \to P1$. The former sends the unique $n$-cell of $1$ to an element $\iota_n
\in P_{\iota_n}$, which we think of as the trivial composition operation of
dimension $n$; whilst the latter assigns to the element
\begin{equation*}
    (\pi \in (T1)_n,\ \theta \in P_\pi,\ \psi \colon \hat \pi \to P1)
\end{equation*}
of $(PP1)_n$ an element $\theta \circ \psi \in P_{\phi \circ \pi}$ (where
$\phi$ is the composite $\rho_1 \psi \colon \hat \pi \to T1$), which we think
of as the composition operation obtained by substituting into $\theta$ the
collection of operations indexed by $\psi$.

Not every globular operad embodies a sensible theory of weak
$\omega$-categories---since, for example, the identity monad on $\cat{GSet}$ is
a globular operad---but~\cite{Batanin1998Monoidal} provides two conditions
which together distinguish those which do: \emph{normalisation} and
\emph{contractibility}. Normalisation is straightforward; it asserts that the
monad $P$ is bijective on objects in the sense that $(PX)_0 \cong X_0$,
naturally in $X$, or equivalently, that the set $P_\star$ is a singleton. The
second condition  is a little more subtle. A globular operad $P$ is said to be
\emph{contractible} if:
\begin{enumerate}[(a)]
\item Given $\pi \in (T1)_1$ and $\theta_1, \theta_2 \in P_\star$, there
    exists an element $\phi \in P_\pi$ with $s(\phi) = \theta_1$ and
    $t(\phi) = \theta_2$;
\item Given $\pi \in (T1)_n$ (for $n > 1$) and $\theta_1, \theta_2 \in
    P_{\partial\pi}$ satisfying $s(\theta_1) = s(\theta_2)$ and
    $t(\theta_1) = t(\theta_2)$, there exists an element $\phi \in P_\pi$
    such that $s(\phi) = \theta_1$ and $t(\phi) = \theta_2$.
\end{enumerate}
Contractibility expresses that that a globular operad has ``enough'' ways of
composing to yield a theory of weak $\omega$-categories. In homotopy-theoretic
terms, a contractible globular operad is a ``deformation'' of the monad~$T$; an
idea which can be made precise using the language of weak factorisation
systems: see~\cite{Garner2008homotopy-theoretic}.
\begin{Defn}
A \emph{weak $\omega$-category} is an algebra for a contractible, normalised,
globular operad: more formally, it is a pair $(P,X)$, where $P$ is a
contractible, normalised, globular operad and $X$ is an algebra for it.
\end{Defn}
\begin{Rk}
Some consideration must be paid to the exact force of the term
\emph{contractible}, which has been used in different ways by different
authors; our usage accords with that of~\cite[Definition
9.1.3]{Leinster2004Higher}. In particular, the reader should carefully
distinguish between the \emph{property} of being contractible described above,
and the corresponding \emph{structure} of being equipped with a contraction.
\end{Rk}

We now turn from the definition of weak $\omega$-category to that of weak
$\omega$-groupoid. For this we will require the coinductive notion of
\emph{equivalence} in a weak $\omega$-category.
\begin{Defn}\label{wkequiv} Let $(P,X)$ be a weak $\omega$-category. An \emph{equivalence} $x
\simeq y$ between parallel $n$-cells $x, y$ is given by:
\begin{itemize}
\item $n+1$-cells $f \colon x \to y$ and $g \colon y \to x$;
\item Equivalences $\eta \colon g \circ f \simeq \id_x$ and $\epsilon
    \colon f \circ g \simeq \id_y$.
\end{itemize}
We say that an $(n+1)$-cell $f \colon x \to y$ is \emph{weakly invertible} if
it participates in an equivalence $(f, g, \eta, \epsilon)$.
\end{Defn}
In order for this definition to make sense, we must determine what is meant by
the expressions ``$\id_x$'', ``$\id_y$'', ``$g \circ f$'' and ``$f \circ g$''
appearing in it, which we may do as follows. First, for each $n \geqslant 1$,
we define the pasting diagrams $0_n$ and $2_n \in (T1)_n$ to be given by
\begin{equation*}
    0_n \defeq \underbrace{(\cdots (}_{n\ \text{times}})\cdots) \qquad \text{and} \qquad
    2_n \defeq \underbrace{(\cdots (}_{n\ \text{times}}\!\!\star, \star)\cdots)\
    \text.
\end{equation*}
Next, if $P$ is a normalised, contractible globular operad, then we define a
\emph{system of compositions} for $P$ to be a choice, for each $n \geqslant 1$,
of operations $i_n \in P_{0_n}$ and $m_n \in P_{2_n}$. Note that the
contractibility of $P$ ensures that it will possess at least one system of
compositions. Finally, if we are given a system of compositions and a
$P$-algebra $X$, then we define the functions
\begin{equation*}
    \id_{(\thg)} \colon X_{n-1} \to X_n \qquad \text{and} \qquad \circ \colon X_{n+1} \mathbin{{}_s\!\times_t} X_{n+1} \to
    X_n
\end{equation*}
to be the interpretations of the operations $i_n$ and $m_n$ respectively. This
allows us to give meaning to the undefined expressions appearing in
Definition~\ref{wkequiv}.
\begin{Defn}
A weak $\omega$-category $(P,X)$ is a \emph{weak $\omega$-groupoid} if every
cell of $X$ is weakly invertible with respect to every system of compositions
on $P$.
\end{Defn}
It will be convenient to give a more elementary reformulation of the notion of
weak $\omega$-groupoid due to Cheng~\cite{Cheng2007w-category}. This is given
in terms of \emph{duals}. If $f \colon x \to y$ is an $n$-cell (for $n
\geqslant 1$) in a weak $\omega$-category, then a \emph{dual} for $f$ is an
$n$-cell $f^\ast \colon y \to x$ together with $(n+1)$-cells $\eta \colon \id_x
\to f^\ast \circ f$ and $\epsilon \colon f \circ f^\ast \to \id_y$, subject to
no axioms. Again, this definition is to be interpreted with respect to some
given system of compositions.
\begin{Prop}\label{characterisation}
A weak $\omega$-category is a weak $\omega$-groupoid if and only if, with
respect to every system of compositions, every cell has a dual.
\end{Prop}
\begin{proof}
By coinduction.
\end{proof}

\subsection{Martin-L\"of type theory}\label{mltt}
By \emph{intensional Martin-L\"of type theory}, we mean the logical calculus
set out in Part~II of~\cite{Nordstrom1990Programming}. We now summarise this
calculus. It has four basic forms of judgement: $A \ \ty$ (``$A$ is a type'');
$a \in A$ (``$a$ is an element of the type $A$''); $A = B \ \ty$ (``$A$ and $B$
are definitionally equal types''); and $a = b \in A$ (``$a$ and $b$ are
definitionally equal elements of the type $A$''). These judgements may be made
either absolutely, or relative to a context $\Gamma$ of assumptions, in which
case we write them as
\begin{equation*}
  (\Gamma) \ A\ \ty\text, \qquad
  (\Gamma) \ a \in A\text, \qquad
  (\Gamma) \ A = B\ \ty \qquad \text{and} \qquad
  (\Gamma) \ a = b\in A
\end{equation*}
respectively. Here, a \emph{context} is a list $\Gamma = (x_1 \in A_1\c x_2 \in
A_2\c \dots\c x_n \in A_{n-1})$, wherein each $A_i$ is a type relative to the
context $(x_1 \in A_1\c \dots\c x_{i-1} \in A_{i-1})$. There are now some
rather natural requirements for well-formed judgements: in order to assert that
$a \in A$ we must first know that $A \ \ty$; to assert that $A = B \ \ty$ we
must first know that $A \ \ty$ and $B \ \ty$; and so on. We specify intensional
Martin-L\"of type theory as a collection of inference rules over these forms of
judgement. Firstly we have the \emph{equality rules}, which assert that the two
judgement forms $A = B \ \ty$ and $a = b \in A$ are congruences with respect to
all the other operations of the theory; then we have the \emph{structural
rules}, which deal with weakening, contraction, exchange and substitution; and
finally, the \emph{logical rules}, which specify the type-formers of our
theory, together with their introduction, elimination and computation rules.
For the purposes of this paper, we require only the rules for the identity
types, which we list in Table \ref{fig1}. We commit the usual abuse of notation
in leaving implicit an ambient context $\Gamma$ common to the premisses and
conclusions of each rule, and omitting the rules expressing stability under
substitution in this ambient context. Let us remark also that in the rules
$\Id\textsc{-elim}$ and $\Id\textsc{-comp}$ we allow the type $C$ over which
elimination is occurring to depend upon an additional contextual parameter
$\Delta$. Were we to add $\Pi$-types (dependent products) to our calculus, then
these rules would be equivalent to the usual identity type rules. However, in
the absence of $\Pi$-types, this extra parameter is essential to derive all but
the most basic properties of the identity type.

\begin{table}
\vskip\baselineskip \emph{Identity types}
\begin{equation*}
    \inferrule*[right=$\Id$-form;]{A\ \ty \\ a, b \in A}{\Id_A(a, b) \ \ty} \qquad
    \inferrule*[right=$\Id$-intro;]{A\ \ty \\ a \in A}{\r(a) \in \Id_A(a, a)}
\end{equation*}
\ \begin{equation*}
\inferrule*[right=$\Id$-elim;]{\big(x, y \in A \c p \in \Id_A(x, y) \c \Delta(x,y,p)\big)\  C(x, y, p) \ \ty\\
  \big(x \in A \c \Delta(x, x, \r(x))\big)\  d(x) \in C(x, x, \r(x))\\
  a, b \in A \\ p \in \Id_A(a, b)}
  {\big(\Delta(a,b,p)\big)\  \J_{d}(a, b, p) \in C(a, b, p)}
\end{equation*}
\
\begin{equation*}
\inferrule*[right=$\Id$-comp.]{\big(x, y \in A \c p \in \Id_A(x, y) \c \Delta(x,y,p)\big)\  C(x, y, p) \ \ty\\
  \big(x \in A \c \Delta(x, x, \r(x))\big)\  d(x) \in C(x, x, \r(x))\\
  a \in A}
  {\big(\Delta(a,a,\r(a))\big)\ \J_{d}(a,a,\r(a)) = d(a) \in C(a, a, \r(a))}
\end{equation*}
\caption{Identity type rules}\label{fig1}
\end{table}

We now establish some further notational conventions. Where it improves clarity
we may omit brackets in function applications, writing $hgfx$ in place of
$h(g(f(x)))$, for example. We may drop the subscript $A$ in an identity type
$\Id_A(a, b)$ where no confusion seems likely to occur. Given $a, b \in A$, we
may say that $a$ and $b$ are \emph{propositionally equal} to indicate that the
type $\Id(a, b)$ is inhabited. We will also make use of \emph{vector notation}
in the style of~\cite{DeBruijn1991Telescopic}. Given a context $\Gamma = (x_1
\in A_1, \dots, x_n \in A_n)$, we may abbreviate a series of judgements:
\begin{equation*}
a_1 \in A_1\text, \qquad a_2 \in A_2(a_1)\text, \qquad \dots \qquad a_n \in
A_n(a_1, \dots, a_{n-1})\text,
\end{equation*}
as $\vec a \in \Gamma$, where $\vec a
\defeq (a_1, \dots, a_n)$. We may also use this notation to abbreviate sequences of
hypothetical elements; so, for example, we may specify a dependent type in
context $\Gamma$ as $(\vec x \in \Gamma)\ A(\vec x) \ \ty$. We will also make
use of~\cite{DeBruijn1991Telescopic}'s notion of \emph{telescope}. Given
$\Gamma$ a context as before, this allows us to abbreviate the series of
judgements
\begin{equation*}
\begin{aligned}
(\vec x \in \Gamma) & \ B_1(\vec x) \ \ty\text,\\
(\vec x \in \Gamma \c y_1 \in B_1) & \ B_2(\vec x, y_1) \ \ty\text,\\
& \vdots \\
(\vec x \in \Gamma \c y_1 \in B_1 \c \dots \c y_{m-1} \in B_{m-1}) & \ B_m(\vec
x, y_1, \dots y_{m-1}) \ \ty
\end{aligned}
\end{equation*}
as $
  (\vec x \in \Gamma)\ \Delta(\vec x) \ \ctxt
$, where $\Delta(\vec x) \defeq (y_1 \in B_1(\vec x)\c y_2 \in B_2(\vec x,
y_1)\c \dots )$. We say that $\Delta$ is a \emph{context dependent upon
$\Gamma$}. Given such a dependent context, we may abbreviate the series of
judgements
\begin{equation*}
\begin{gathered}
(\vec x \in \Gamma) \ f_1(\vec x) \in B_1(\vec x)\\
\vdots \\
(\vec x \in \Gamma) \ f_m(\vec x) \in B_m(\vec x, f_1(\vec x), \dots,
f_{m-1}(\vec x))\text,
\end{gathered}
\end{equation*}
as $    (\vec x \in \Gamma)\ \vec f(\vec x) \in \Delta(\vec x)$, and say that
$\vec f$ is a \emph{dependent element} of $\Delta$. We can similarly assign a
meaning to the judgements
  $(\vec x \in \Gamma)\ \Delta(\vec x) = \Theta(\vec x) \ \ctxt$ and $(\vec x \in \Gamma)\ \vec f(\vec x) = \vec g(\vec x) \in \Delta(\vec
  x)$,
expressing the definitional equality of two dependent contexts, and the
definitional equality of two dependent elements of a dependent context.

Let us now recall some basic facts about categorical models of type theory. For
a more detailed treatment the reader could refer 
to~\cite{Jacobs1999Categorical,Pitts2000Categorical}, for example. If $\mathbb
T$ is a dependently typed calculus admitting each of the rules described above,
then we may construct from it a category $\C_\mathbb T$ known as the
\emph{classifying category} of $\mathbb T$. Its objects are contexts
$\Gamma$,~$\Delta$,~\dots, in $\mathbb T$, considered modulo definitional
equality (so we identify $\Gamma$ and $\Delta$ whenever $\Gamma = \Delta \
\ctxt$ is derivable); and its maps $\Gamma \to \Delta$ are \emph{context
morphisms}, which are judgements \mbox{$(\vec x \in \Gamma) \ \vec f (\vec x)
\in \Delta$} considered modulo definitional equality. The identity map on
$\Gamma$ is given by $(\vec x \in \Gamma) \ \vec x \in \Gamma$; whilst
composition is given by substitution of terms. Now, for any judgement $(\vec x
\in \Gamma)\ A(\vec x) \ \ty$ of $\mathbb T$, there is a distinguished context
morphism
\begin{equation*}
    (\vec x \in \Gamma\c y \in A(\vec x)) \to (\vec x \in \Gamma)
\end{equation*}
which sends $(\vec x, y)$ to $\vec x$. We call morphisms of $\C_\mathbb T$ of
this form \emph{basic dependent projections}. By a \emph{dependent projection},
we mean any composite of zero or more basic dependent projections. An important
property of dependent projections is that they are stable under pullback, in
the sense that for every $(\vec x \in \Gamma)\ A(\vec x) \ \ty$ and context
morphism $f \colon \Delta \to \Gamma$, we may show the square
\begin{equation*}
  \cd[@C+1em]{
    \big(\vec w \in \Delta\c y \in A(f(w))\big) \ar[d]_{p'} \ar[r] &
    \big(\vec x \in \Gamma\c y \in A(x)\big) \ar[d]^{p} \\
    \Delta \ar[r]_f &
    \Gamma
  }\ \text,
\end{equation*}
wherein the uppermost arrow sends $(w,y)$ to $(fw, y)$, to be a pullback in
$\C_\mathbb T$. Let us now recall from~\cite{Gambino2008identity} a second
class of maps in $\C_\mathbb T$ which will play an important role in this
paper. A context morphism $f \colon \Gamma \to \Delta$ is said to be an
\emph{injective equivalence} if it validates type-theoretic rules:
\begin{equation*}
\inferrule{(\vec y \in \Delta)\ \Lambda(\vec y) \ \ctxt\\
  (\vec x \in \Gamma)\ \vec d(\vec x) \in \Lambda(\vec f(\vec x))\\
  \vec b \in \Delta}
  {\mathrm E_{\vec d}(\vec b) \in \Lambda(\vec b)}
\end{equation*}
and
\begin{equation*}
\inferrule{(\vec y \in \Delta)\ \Lambda(\vec y) \ \ctxt\\
  (\vec x \in \Gamma)\ \vec d(\vec x) \in \Lambda(\vec f(\vec x))\\
  \vec a \in \Gamma}
  {\mathrm E_{\vec d}(\vec f(\vec a)) = \vec d(\vec a) \in \Lambda(\vec f(\vec a))}\ \text.
\end{equation*}
The name is motivated by the groupoid model of type theory, wherein the
injective equivalences are precisely the injective groupoid equivalences.
Intuitively, a morphism $f \colon \Gamma \to \Delta$ is an injective
equivalence just when every (dependent) function out of $\Delta$ is determined,
up to propositional equality, by its restriction to $\Gamma$. The leading
example of an injective equivalence is given by the context morphism $A \to
(x,y \in A, p \in \Id(x,y))$ sending $x$ to $(x, x, \r x)$. That this map is an
injective equivalence is precisely the content of the $\Id$-elimination and
computation rules. Diagramatically, a map $f$ is an injective equivalence if
for every commutative square of the form
\begin{equation*}
\cd{
  \Gamma \ar[d]_h \ar[r]^-d & (\Delta, \Lambda) \ar[d]^{p} \\
  \Delta \ar@{=}[r] & \Delta
}
\end{equation*}
with $p$ a dependent projection, we may find a diagonal filler $E_d \colon
\Delta \to (\Delta, \Lambda)$ making both induced triangles commute. By the
stability of dependent projections under pullback, this is equivalent with the
property that we should be able to find fillers for all commutative squares of
the form
\begin{equation}\label{commute}
\cd{
  \Gamma \ar[d]_h \ar[r]^-d & (\Phi, \Lambda) \ar[d]^{p} \\
  \Delta \ar[r]_k & \Phi
}
\end{equation}
again with $p$ a dependent projection. See~\cite[Section
5]{Gambino2008identity} for an elementary characterisation of the class of
injective equivalences.

\section{The main result}\label{sec3}
\subsection{An overview of the proof}\label{overview}
We are now ready to begin the proof of our main result: that if $\mathbb T$ is
a dependently typed calculus admitting each of the rules described in
Section~\ref{mltt}, then each type $A$ therein gives rise to a weak
$\omega$-groupoid whose objects are elements of $A$, and whose higher cells are
elements of the iterated identity types on $A$. In fact, we will be able to
prove a stronger result: that $A$ provides the ``type of objects'' for a weak
$\omega$-groupoid which is, in a suitable sense, \emph{internal} to $\mathbb
T$.

As explained in the Introduction, we will give our proof twice: once
informally, using a type-theoretic language, and once formally, using an
axiomatic categorical framework which captures just those aspects of the type
theory which allow the proof to go through. In this Section, we give the
informal proof. We shall concentrate in the first instance on constructing a
weak $\omega$-category, and defer the question of whether or not it is a weak
$\omega$-groupoid until the formal proof.

We begin by defining what we mean by a weak $\omega$-category internal to a
type theory $\mathbb T$. More specifically, given some globular operad $P$, we
define a notion of $P$-algebra internal to $\mathbb T$. The underlying data for
such a $P$-algebra is a \emph{globular context} $(\Delta)\ \Gamma \in \mathbb
T$; which is a sequence of judgements
\begin{align*}
& (\Delta)\ \Gamma_0 \ \ctxt\\
(\Delta, \vec x, \vec y \in \Gamma_0) \ & \Gamma_1(\vec x, \vec y) \ \ctxt\\
\big(\Delta, \vec x, \vec y \in \Gamma_0\c \vec p, \vec q \in \Gamma_1(\vec x,
\vec
y)\big) \ & \Gamma_2(\vec x, \vec y, \vec p, \vec q) \ \ctxt\\
& \vdots
\end{align*}
Just as globular sets have a coinductive characterisation, so too do globular
contexts: to give a globular context $(\Delta)\ \Gamma$ is to give a context
$(\Delta)\ \Gamma_0$ together with a globular context $(\Delta\c \vec x, \vec y
: \Gamma_0)\ \Gamma_{+1}(\vec x, \vec y)$. In order to define the operations
making a globular context $\Gamma$ (where henceforth we simplify the notation
by  omitting the precontext $\Delta$) into a $P$-algebra, we first define for
each pasting diagram $\pi \in (T1)_n$ the context $\Gamma^\pi$ consisting of
``$\pi$-indexed elements of $\Gamma$''. This is done by induction on $\pi$:
\begin{itemize}
\item If $\pi = \star$, then $\Gamma^\pi \defeq \Gamma_0$;
\item If $\pi = (\pi_1, \dots, \pi_k)$, then $\Gamma^\pi$ is the context
\begin{equation*}
  \big(\vec x_0, \dots, \vec x_k \in \Gamma_0 \c \vec y_1 \in \Gamma_{+1}(\vec x_0, \vec x_1)^{\pi_1}, \dots, \vec y_k \in \Gamma_{+1}(\vec x_{k-1}, \vec
  x_k)^{\pi_k}\big)\ \text.
\end{equation*}
\end{itemize}
For example, if $\pi$ is the pasting diagram~\eqref{shape}, then the context
$\Gamma^\pi$ is given by:
\begin{align*}
    \big(\ &\vec x_0, \vec x_1, \vec x_2 \in \Gamma_0\c\\
    &\ \vec s, \vec t \in \Gamma_{1}(\vec x_0, \vec x_1)\c \vec \alpha \in \Gamma_2(x_0, x_1, \vec s, \vec t)\c\\
    &\ \vec u, \vec v, \vec w \in \Gamma_{1}(\vec x_1, \vec x_2)\c \vec \beta \in \Gamma_2(x_1, x_2, \vec u, \vec
    v)\c \vec \gamma \in \Gamma_2(x_1, x_2, \vec v, \vec
    w)\ \big)
\end{align*}
whilst if $\pi \in (T1)_n$ is the element $\iota_n = (\cdots (\star) \cdots )$,
then $\Gamma^{\iota_n}$ is the context
\begin{equation*}
    \big(\vec x_0, \vec y_0 \in \Gamma_0\c \vec x_{1}, \vec y_{1} \in
    \Gamma_{1}(\vec x_0, \vec y_0)\c \dots\c \vec x_n \in
    \Gamma_n(\vec x_0, \vec y_0, \dots, \vec x_{n-1}, \vec y_{n-1})\big)
\end{equation*}
indexing the totality of the $n$-cells of $\Gamma$. Now to give a $P$-algebra
structure on the globular context $\Gamma$ will be to give, for every $\pi \in
(T1)_n$ and $\theta \in P_\pi$, a context morphism
\begin{equation*}
    [\theta] \colon \Gamma^\pi \to \Gamma^{\iota_n}
\end{equation*}
interpreting the operation $\theta$, subject to the following axioms. Firstly,
the interpretations should be compatible with source and target, which is to
say that diagrams of the form
\begin{equation*}
   \cd{
     \Gamma^\pi \ar[r]^{[\theta]} \ar[d]_{\sigma} &
     \Gamma^{\iota_n} \ar[d]^{\sigma} \\
     \Gamma^{\partial\pi} \ar[r]_{[s\theta]} &
     \Gamma^{\iota_{n-1}}
   } \qquad \text{and} \qquad
   \cd{
     \Gamma^\pi \ar[r]^{[\theta]} \ar[d]_{\tau} &
     \Gamma^{\iota_n} \ar[d]^{\tau} \\
     \Gamma^{\partial\pi} \ar[r]_{[t\theta]} &
     \Gamma^{\iota_{n-1}}
   }
\end{equation*}
should commute; here, $\sigma, \tau \colon \Gamma^\pi \to \Gamma^{\partial
\pi}$ are source and target projections defined by a further straightforward
induction over $\pi$. Secondly, the trivial pasting operations should have a
trivial interpretation; which is to say that
\begin{equation*}
  [\iota_n] = \id_{\Gamma^{\iota_n}} \colon \Gamma^{\iota_n} \to \Gamma^{\iota_n}\ \text.
\end{equation*}
Thirdly, the interpretation of a composite $[\theta \circ \psi]$ should be
``given by the composite of $[\theta]$ with $[\psi]$'', in the sense that the
following diagram commutes:
\begin{equation*}
\cd{
    \Gamma^{\pi \circ \phi} \ar[r]^{[\psi]} \ar[dr]_{[\theta \circ \psi]} &
    \Gamma^{\pi} \ar[d]^{[\theta]} \\
    & \Gamma^{\iota_n}
}\ \text.
\end{equation*}
This is not yet entirely formal, because we have not indicated how the map
$[\psi] \colon \Gamma^{\pi \circ \phi} \to \Gamma^{\pi}$ should be defined.
Intuitively, it is the morphism which applies simultaneously the
interpretations of the operations indexed by $\psi \colon \hat \pi \to P1$; but
it is not immediately clear how to make this precise. We will do so in
Section~\ref{proof1} below, using Michael Batanin's machinery of \emph{monoidal
globular categories}~\cite{Batanin1998Monoidal}. A general result from this
theory allows us to associate to the globular context $\Gamma$ a particular
globular operad $[\Gamma, \Gamma]$---the \emph{endomorphism operad} of
$\Gamma$---which is such that we may \emph{define} $P$-algebra structures on
$\Gamma$ to be globular operad morphisms $P \to [\Gamma, \Gamma]$. This
 operad $[\Gamma, \Gamma]$ has as operations of shape $\pi$, all
serially commutative diagrams
\begin{equation}\label{operation}
   \cd{
     \Gamma^\pi \ar[r]^{f_n} \ar@<3pt>[d]^{\tau} \ar@<-3pt>[d]_{\sigma} &
     \Gamma^{\iota_n} \ar@<3pt>[d]^{\tau} \ar@<-3pt>[d]_{\sigma} \\
     \Gamma^{\partial\pi} \ar@<3pt>[r]^{f_{n-1}} \ar@<-3pt>[r]_{g_{n-1}} \ar@<3pt>[d]^{\tau} \ar@<-3pt>[d]_{\sigma} &
     \Gamma^{\iota_{n-1}} \ar@<3pt>[d]^{\tau} \ar@<-3pt>[d]_{\sigma} \\
     \vdots \ar@<3pt>[d]^{\tau} \ar@<-3pt>[d]_{\sigma} & \vdots \ar@<3pt>[d]^{\tau}
     \ar@<-3pt>[d]_{\sigma} \\
     \Gamma^\star \ar@<3pt>[r]^{f_{0}} \ar@<-3pt>[r]_{g_{0}} & \Gamma^\star
   }
\end{equation}
of context morphisms. The source and target functions $[\Gamma, \Gamma]_\pi \to
[\Gamma, \Gamma]_{\partial \pi}$ send such a diagram to its subdiagram headed
by $f_{n-1}$, respectively $g_{n-1}$; the identity operation $\iota_n \in
[\Gamma, \Gamma]_{\iota_n}$ has each $f_i$ and $g_i$ given by an identity map;
whilst to describe substitution of operations in $[\Gamma, \Gamma]$ is
precisely the problem that we encountered above, and that which Batanin's
machinery solves. It is easy to see that a map of globular operads $P \to
[\Gamma, \Gamma]$ encodes exactly the structure of an internal $P$-algebra
sketched above.

We may now give a precise statement of the main result. Given a type theory
$\mathbb T$ admitting the rules of Section~\ref{mltt} and a type $A \in \mathbb
T$, we will construct a normalised, contractible, globular operad $P$ such that
the globular context $\underline A$ given by
\begin{align*}
& A \ \ctxt\\
(x, y \in A) \ & \Id_A(x, y) \ \ctxt\\
\big(x, y \in A\c  p, q \in \Id_A(x, y)\big) \ & \Id_{\Id_A(x,y)}(p, q) \ \ctxt\\
& \vdots
\end{align*}
admits an internal $P$-algebra structure. Now, it is straightforward to find an
operad for which $\underline A$ is an algebra---namely, the endomorphism operad
$[\underline A, \underline A]$, with algebra structure given by the identity
morphism $[\underline A, \underline A] \to [\underline A, \underline A]$---but
this does not help us, since there is no reason to expect this operad to be
either normalised or contractible. However, it comes rather close to being
contractible, in a sense which we will now explain. For $[\underline A,
\underline A]$ to be contractible would be for us to ask that, for every
serially commutative diagram
\begin{equation}\label{tocomplete}
   \cd{
     {\underline A}^\pi \ar@<3pt>[d]^{\tau} \ar@<-3pt>[d]_{\sigma} &
     {\underline A}^{\iota_n} \ar@<3pt>[d]^{\tau} \ar@<-3pt>[d]_{\sigma} \\
     {\underline A}^{\partial\pi} \ar@<3pt>[r]^{f_{n-1}} \ar@<-3pt>[r]_{g_{n-1}} \ar@<3pt>[d]^{\tau} \ar@<-3pt>[d]_{\sigma} &
     {\underline A}^{\iota_{n-1}} \ar@<3pt>[d]^{\tau} \ar@<-3pt>[d]_{\sigma} \\
     \vdots \ar@<3pt>[d]^{\tau} \ar@<-3pt>[d]_{\sigma} & \vdots \ar@<3pt>[d]^{\tau}
     \ar@<-3pt>[d]_{\sigma} \\
     {\underline A}^\star \ar@<3pt>[r]^{f_{0}} \ar@<-3pt>[r]_{g_{0}} & {\underline A}^\star
   }
\end{equation}
of context morphisms, we could find a map $\underline A^\pi \to \underline
A^{\iota_n}$ completing it to a diagram like~\eqref{operation}. Let us consider
in particular the case where $\pi$ is the pasting diagram of~\eqref{shape}.
Here, to give the data of~\eqref{tocomplete} is to give judgements
\begin{equation}\label{judgements1}
\begin{aligned}
    (x \in A)\ & f_0(x) \in A \\
    (x \in A)\ & g_0(x) \in A \\
    \big(x, y, z \in A\c p \in \Id(x,y), q \in \Id(y,z)\big)\ & f_1(x,y,z,p,q) \in \Id(f_0(x), g_0(z)) \\
    \big(x, y, z \in A\c p \in \Id(x,y), q \in \Id(y,z)\big)\ & g_1(x,y,z,p,q) \in \Id(f_0(x), g_0(z))
\end{aligned}
\end{equation}
whilst to give its completion $f_2 \colon \underline A^\pi \to \underline
A^{\iota_2}$ would be to give a judgement
\begin{align*}
    \big(\ &x, y, z \in A\c\\
    &\ s, t \in \Id(x, y)\c \alpha \in \Id(s,t)\c\\
    &\ u, v, w \in \Id(y,z)\c \beta \in \Id(u,v)\c \gamma \in \Id(v,w)\ \big) \\
    & \qquad f_2(x,y,z,s,t,\alpha,u,v,w,\beta,\gamma) \in \Id\big(f_1(x,y,z,s,u),
    g_1(x,y,z,t,w)\big)\ \text.
\end{align*}
We might attempt to obtain such a judgement by repeated application of the
identity type elimination rule. Indeed, by $\Id$-elimination on $\alpha$ it
suffices to consider the case where $s = t$ and $\alpha = \r(s)$; and by
$\Id$-elimination on $\gamma$ and $\beta$, it suffices to consider the case
where $u=v=w$ and $\gamma = \beta = \r(u)$. Thus it suffices to find a term
\begin{multline*}
    \big(x, y, z \in A\c s \in \Id(x, y)\c u \in \Id(y,z)\big)\\ f'_2(x,y,z,s,u) \in \Id\big(f_1(x,y,z,s,u), g_1(x,y,z,s,u)\big)\ \text.
\end{multline*}
But now by $\Id$-elimination on $s$ and on $u$, it suffices to consider the
case where $x = y = z$ and $s = u = \r(x)$; so that it even suffices to find a
term
\begin{equation}\label{fff}
    (x \in A)\ f''_2(x) \in \Id\big(f_1(x,x,x,\r x,\r x), g_1(x,x,x,\r x,\r x)\big)\ \text.
\end{equation}
Yet here we encounter the problem that $f_1$ and $g_1$, being arbitrarily
defined, need not agree at $(x,x,x,\r x, \r x)$, so that there is in general no
reason for a term like~\eqref{fff} to exist. However, there is a
straightforward way of removing this obstruction: we restrict attention to
those operations of shape $\pi$ which, when applied to a term consisting solely
of reflexivity proofs, yield another reflexivity proof. We may formalise this
as follows. For each $\pi \in (T1)_n$, we define, by induction on $\pi$, a
pointing $r_\pi \colon A \to \underline A^\pi$:
\begin{itemize}
\item If $\pi = \star$, then $r_\star \defeq \id \colon A \to A$;
\item If $\pi = (\pi_1, \dots, \pi_k)$, then $r_\pi$ is the context
    morphism
\begin{equation*}
  (x \in A)\ (\underbrace{x, \dots, x}_{{\text{$k$ times}}}, r_{\pi_1}(\r x), \dots, r_{\pi_k}(\r x)) \in \underline A^\pi\ \text.
\end{equation*}
\end{itemize}
In our example, if the judgements in~\eqref{judgements1} commuted with the
$A$-pointings, then we would have that $f_1(x,x,x,\r x, \r x) = g_1(x, x, x, \r
x, \r x) = \r (x) \in \Id(x, x)$, so that in~\eqref{fff} we could define
\begin{equation*}
    (x \in A)\ f''_2(x) \defeq \r (\r x) \in \Id(\r x, \r x)
\end{equation*}
and in this way obtain by repeated $\Id$-elimination the desired completion
$f_2 \colon \underline A^\pi \to \underline A^{\iota_2}$.
Motivated by this, we define the sub-operad $P \subset [\underline A,
\underline A]$ to have as its operations of shape $\pi$, those diagrams of the
form~\eqref{operation} in which each $f_i$ and $g_i$ commutes with the
$A$-pointings just defined. Again, it is intuitively clear that this defines a
sub-operad---which is to say that the operations with this property are closed
under identities and substitution---but to prove this requires a second
excursion into the theory of monoidal globular categories: one which for the
purposes of the present section, we omit. However, we claim further that $P$ is
both normalised and contractible. This will then prove our main result, since
the globular context $\underline A$ is a $P$-algebra---as witnessed by the map
of globular operads $P \hookrightarrow [\underline A, \underline A]$---so that
we will have shown the globular context $\underline A$ to be an algebra for a
normalised, contractible, globular operad $P$, and hence a weak
$\omega$-category.

Now, to show $P$ normalised is trivial, since its operations of shape $\star$
are those context morphisms $A \to A$ which commute with the pointing $\id_A
\colon A \to A$, and there is of course only one such. On the other hand, we
see that it is contractible through a generalisation of the argument given in
the example above. The only part requiring some thought is how to describe
generically the process of repeatedly applying $\Id$-elimination. The key to
doing this is to prove by induction on $\pi$ that each of the pointings $r_\pi
\colon A \to \underline A^\pi$ is an injective equivalence in the sense defined
in Section~\ref{mltt}. The injective equivalence structure now encodes the
process of repeated $\Id$-elimination. Using this, we may show $P$ contractible
as follows. Suppose we are given a diagram like~\eqref{tocomplete} where each
$f_i$ and $g_i$ commutes with the $A$-pointings. We let $B\underline
A^{\iota_n}$ denote the context obtained from $\underline A^{\iota_n}$ by
removing its final variable, and let $p \colon \underline A^{\iota_n} \to
B\underline A^{\iota_n}$ denote the corresponding dependent projection. Then we
have a commutative square
\begin{equation*}
\cd[@C+2em]{
  A \ar[d]_{r_\pi} \ar[r]^-{r_{\iota_n}} & \underline A^{\iota_n} \ar[d]^{p} \\
  \underline A^\pi \ar[r] & B\underline A^{\iota_n}
} \ \text,
\end{equation*}
where the lower arrow is obtained by applying first the projection $\underline
A^\pi \to B\underline A^\pi$, and then the maps $f_{n-1}$ and $g_{n-1}$.
Commutativity obtains by virtue of the fact that $f_{n-1}$ and $g_{n-1}$
commute with the pointings; and so, because $r_\pi$ is an injective equivalence
and $p$ a dependent projection, we can find a diagonal filler, which will be
the required map $f_n \colon \underline A^\pi \to \underline A^{\iota_n}$.

\subsection{An axiomatic framework}\label{axiomatic}
We now wish to make rigorous the above proof; and as we have already mentioned,
we shall do so not in an explicitly type-theoretic manner, but rather within an
axiomatic categorical framework. In this Section, we describe this framework
and give the intended type-theoretic interpretation.
\begin{Defn}
A category $\C$ is an \emph{identity type category} if it comes equipped with
two classes of maps $\I, \P \subset \mor \C$ satisfying the following axioms:
\begin{description}
\item[Empty] $\C$ has a terminal object $1$, and for all $A \in \C$, the
    unique map $A \to 1$ is a $\P$-map.
\item[Composition] The classes of $\P$-maps and $\I$-maps contain the
    identities and are closed under composition.
\item[Stability] Pullbacks of $\P$-maps along arbitrary maps exist, and are
    again $\P$-maps.
\item[Frobenius] The pullback of an $\I$-map along a $\P$-map is an
    $\I$-map.
\item[Orthogonality] For every commutative square
\begin{equation}\label{orthog}
\cd{
  A \ar[d]_i \ar[r]^-f & C \ar[d]^{p} \\
  B \ar[r]_g & D
}
\end{equation}
with $i \in \I$ and $p \in \P$, we can find a diagonal filler $j \colon B
\to C$ such that $ji = f$ and $pj = g$.
\item[Identities] For every $\P$-map $p \colon C \to D$, the diagonal map
    $\Delta \colon C \to C \times_D C$ has a factorisation
\begin{equation*}
  \Delta = C \xrightarrow{\quad r\quad } \Id(C) \xrightarrow{\quad e\quad } C \times_D C
\end{equation*}
where $r \in \I$ and $e \in \P$.
\end{description}
\end{Defn}
We make two remarks concerning this definition. Firstly, by (Empty) and
(Stability), any identity type category will have finite products, and product
projections will be $\P$-maps. Secondly, in order to verify (Orthogonality), it
suffices, by (Stability), to do so only in those cases where the map along the
bottom of~\eqref{orthog} is an identity.

\begin{Prop}
Let $\mathbb T$ be a dependent type theory admitting each of the inference
rules described in Section~\ref{mltt}. Then the classifying category
$\C_\mathbb T$ is an identity type category, where we take $\P$ to be the class
of dependent projections and $\I$ the class of injective equivalences.
\end{Prop}
\begin{proof}
The empty context $(\ )$ provides a terminal object of $\C_\mathbb T$.
(Composition) is immediate from the definitions. (Stability) corresponds to the
possibility of performing type-theoretic substitution. (Frobenius) is shown to
hold in~\cite[Proposition 14]{Gambino2008identity}; it is a categorical
correlate of the fact that we allow an extra contextual parameter $\Delta$ in
the statement of the $\Id$-elimination rule. (Orthogonality) holds by the very
definition of injective equivalence, together with the remark made above.
Finally, (Identities) says something more than that identity types exist---it
says that identity \emph{contexts} exist: which is to say that, for every
dependent context $(\Delta) \ \Gamma \ \ctxt$, we may find a context $(\Delta\c
\vec x, \vec y \in \Gamma) \ \Id_{\Gamma}(\vec x, \vec y) \ \ctxt$ such that
the contextual analogues of the identity type rules are validated. That this is
possible is proven in~\cite[Proposition 3.3.1]{Garner2008Two-dimensional}.
\end{proof}
\noindent We will also require two stability properties of identity type
categories.
\begin{Prop}
Let $\C$ be an identity type category, and $X \in \C$. Then the coslice
category $X / \C$ is also an identity type category, where we take the class of
$\I$-maps (respectively, $\P$-maps) to consist of those morphisms which become
$\I$-maps (respectively, $\P$-maps) upon application of the forgetful functor
$X / \C \to \C$.
\end{Prop}
\begin{Prop}
Let $\C$ be an identity type category, and $X \in \C$. Then the category
$\C_X$, whose objects are $\P$-maps $A \to X$ and whose morphisms are
commutative triangles, is also an identity type category, where we define the
classes of $\I$-maps and $\P$-maps in a manner analogous to that of the
previous Proposition.
\end{Prop}
\noindent The proofs are trivial; the only point of note is that, in the second
instance, we could not take $\C_X$ to be the full slice category $\C/X$, as
then (Empty) would not be satisfied.

\subsection{Internal weak $\omega$-groupoids}\label{proof1}
In this section, we describe the notion of weak $\omega$-groupoid internal to
an identity type category $\C$. We begin by defining internal $P$-algebras for
a globular operad $P$.
\begin{Defn}\label{globctxt}
A \emph{pre-globular context} in $\C$ is a diagram
\begin{equation*}
    \cd[@C-0.5em]{
    \Gamma_0 & \Gamma_1 \ar@<3pt>[l]^-{t} \ar@<-3pt>[l]_-{s} & \Gamma_2 \ar@<3pt>[l]^-{t} \ar@<-3pt>[l]_-{s} & \Gamma_3\ar@<3pt>[l]^-{t} \ar@<-3pt>[l]_-{s} & \,\cdots \ar@<3pt>[l]^-{t} \ar@<-3pt>[l]_-{s}}
\end{equation*}
satisfying the globularity equations $ss = st$ and $ts = tt$. A pre-globular
context is a \emph{globular context} if, for each $n \geqslant 1$, the map
\begin{equation}\label{st}
    (s, t) \colon \Gamma_n \to B_n\Gamma
\end{equation}
is a $\P$-map, where $B_n \Gamma$ is defined as follows. We have $B_1 \Gamma
\defeq \Gamma_0 \times \Gamma_0$, and have $B_{n+1} \Gamma$ given by the pullback
\begin{equation}\label{boundary}
    \cd{
        B_{n+1} \Gamma \ar[r]^{} \ar[d] & \Gamma_{n} \ar[d]^{(s,t)} \\
        \Gamma_{n} \ar[r]_-{(s,t)} & B_{n}\Gamma\ \text.
    }
\end{equation}
\end{Defn}
Observe that requiring~\eqref{st} to be a $\P$-map for $n = 1$ ensures the
existence of the pullback~\eqref{boundary} defining $B_2\Gamma$; which in turn
allows us to require that~\eqref{st} should be a $\P$-map for $n = 2$, and so
on. Once again, we have a coinductive characterisation of globular contexts: to
give a globular context $\Gamma \in \C$ is to give an object $\Gamma_0$
together with a globular context $\Gamma_{+1} \in \C_{\Gamma_0 \times
\Gamma_0}$.
%
%
%
%
%
%

The first step in defining $P$-algebra structure on a globular context $\Gamma$
is to describe the object $\Gamma^\pi$ of ``$\pi$-indexed elements of
$\Gamma$''.
\begin{Defn}\label{idxd}
Let $\Gamma$ be a globular context in $\C$, and let $\pi \in (T1)_n$. We define
the object $\Gamma^\pi \in \C$ by the following induction:
\begin{itemize}
\item If $\pi = \star$ then $\Gamma^\pi \defeq \Gamma_0$.
\item If $\pi = (\pi_1, \dots, \pi_k)$, then we first form the objects
    $(\Gamma_{+1})^{\pi_1}$, \dots, $(\Gamma_{+1})^{\pi_k}$ of
    $\C_{\Gamma_0 \times \Gamma_0}$. This yields a diagram
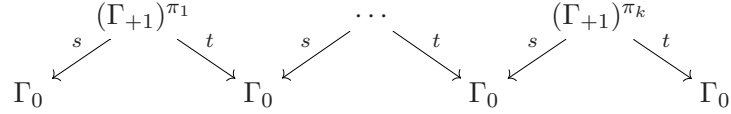
\begin{equation*}
    \cd[@-1em]{
        & (\Gamma_{+1})^{\pi_1} \ar[dl]_s \ar[dr]^t & & \ \ \ \cdots \ \ \ \ar[dl]_s \ar[dr]^t & & (\Gamma_{+1})^{\pi_k} \ar[dl]_s
        \ar[dr]^t \\
        \Gamma_0 & & \Gamma_0 & & \Gamma_0 & & \Gamma_0
    }
\end{equation*}
in $\C$. Note that each $s$ and $t$ is a $\P$-map so that this diagram has
a limit, which we define to be $\Gamma^\pi$.
\end{itemize}
We define maps $\sigma, \tau \colon \Gamma^{\pi} \to \Gamma^{\partial \pi}$ by
a further induction:
\begin{itemize}
\item For $\pi \in (T1)_1$, we have $\Gamma^\pi$ given by the limit of a
    diagram
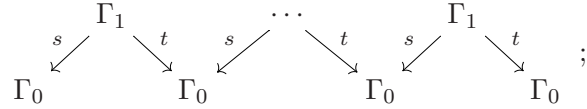
\begin{equation*}
    \cd[@-1em]{
        & \Gamma_1 \ar[dl]_s \ar[dr]^t & & \ \cdots \ \ar[dl]_s \ar[dr]^t & & \Gamma_1 \ar[dl]_s
        \ar[dr]^t \\
        \Gamma_0 & & \Gamma_0 & & \Gamma_0 & & \Gamma_0
    }\ \text;
\end{equation*}
and so we may take $\sigma, \tau \colon \Gamma^\pi \to \Gamma^\star =
\Gamma_0$ to be given by the projections from this limit into the leftmost,
respectively rightmost, copy of $\Gamma_0$.
\item Otherwise, given $\pi = (\pi_1, \dots, \pi_k)$, we first construct
    the morphisms $\sigma, \tau \colon (\Gamma_{+1})^{\pi_i} \to
    (\Gamma_{+1})^{\partial\pi_i}$. These give rise to a diagram
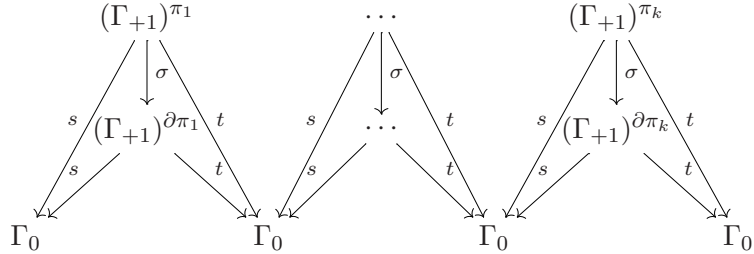
\begin{equation*}
    \cd[@-1em@R+1em]{
        & (\Gamma_{+1})^{\pi_1} \ar[ddl]_s \ar[ddr]^t \ar[d]^\sigma & & \ \ \ \cdots \ \ \ \ar[d]^\sigma \ar[ddl]_s \ar[ddr]^t & & (\Gamma_{+1})^{\pi_k} \ar[ddl]_s
        \ar[ddr]^t \ar[d]^\sigma \\
        & (\Gamma_{+1})^{\partial\pi_1} \ar[dl]_s \ar[dr]^t & & \ \ \ \cdots \ \ \ \ar[dl]_s \ar[dr]^t & & (\Gamma_{+1})^{\partial\pi_k} \ar[dl]_s
        \ar[dr]^t \\
        \Gamma_0 & & \Gamma_0 & & \Gamma_0 & & \Gamma_0
    }
\end{equation*}
and correspondingly for $\tau$. We now take $\sigma, \tau \colon \Gamma^\pi
\to \Gamma^{\partial \pi}$ to be the induced maps from the limit of the
upper subdiagram (which is $\Gamma^\pi$) to the limit of the lower one
(which is $\Gamma^{\partial \pi}$).
\end{itemize}
\end{Defn}
\begin{Prop}\label{endo}
Let $\Gamma \in \C$ be a globular context. Then there is a globular operad
$[\Gamma, \Gamma]$ whose set of operations of shape $\pi$ comprises all
serially commutative diagrams of the form~\eqref{operation}.
\end{Prop}
We will prove this Proposition using Michael Batanin's theory of \emph{monoidal
globular categories}~\cite{Batanin1998Monoidal}. The notion of monoidal
globular category bears the same relationship to that of strict
$\omega$-category as the notion of monoidal category does to that of monoid; in
both cases, the former notion is obtained from the latter by replacing
everywhere sets with categories, functions with functors, and equalities with
coherent natural isomorphisms.
\begin{Defn}
A \emph{monoidal globular category} $\E$ is given by a sequence of categories
and functors
\begin{equation*}
    \cd[@C-0.5em]{
    \E_0 & \E_1 \ar@<3pt>[l]^-{T} \ar@<-3pt>[l]_-{S} & \E_2 \ar@<3pt>[l]^-{T} \ar@<-3pt>[l]_-{S} & \E_3\ar@<3pt>[l]^-{T} \ar@<-3pt>[l]_-{S} & \,\cdots \ar@<3pt>[l]^-{T} \ar@<-3pt>[l]_-{S}}
\end{equation*}
satisfying the globularity equations $SS = ST$ and $TS = TT$, together with,
for each natural number $n$, an identities functor
\begin{equation*}
    Z \colon \E_n \to \E_{n+1}
\end{equation*}
and for each pair of natural numbers $0 \leqslant k < n$, a composition functor
\begin{equation*}
    \otimes_k \colon \E_n \times_k \E_n \to \E_n
\end{equation*}
where $\E_n \times_k \E_n$ denotes the pullback
\begin{equation*}
    \cd{
        \E_n \times_k \E_n \ar[r] \ar[d] &
        \E_n \ar[d]^{S^{n-k}} \\
        \E_n \ar[r]_{T^{n-k}} & \E_{k} \ \text.
    }
\end{equation*}
In addition, there are given invertible natural transformations witnessing:
\begin{itemize}
\item Associativity:
\begin{equation*}
    \alpha_{n,k} \colon A \otimes_k (B \otimes_k C) \cong (A \otimes_k B)
    \otimes_k C
\end{equation*}
\item Unitality:
\begin{equation*}
    \lambda_{n} \colon Z^{n-k}T^{n-k} A \otimes_k A \cong A \quad \text{and}
    \quad
    \rho_{n} \colon A \otimes_k Z^{n-k}S^{n-k} A  \cong A
\end{equation*}
\item Interchange:
\begin{equation*}
    \chi_{n,k,l} \colon (A \otimes_k B) \otimes_l (C \otimes_k D) \cong
(A \otimes_l C) \otimes_k (B \otimes_l D) \quad \text{(for $k < l$).}
\end{equation*}
\end{itemize}
These data are required to satisfy a number of coherence axioms, which the
reader may find in~\cite[Definition 2.3]{Batanin1998Monoidal}.
\end{Defn}
Just as monoidal categories provide a general environment within which we can
speak of monoids, so monoidal globular categories provide a general environment
within which we can speak of algebras for a globular operad. The underlying
data for an algebra in this general setting is given as follows:
\begin{Defn}
A \emph{globular object} $X$ in a monoidal globular category $\E$ is given by a
sequence of objects $X_i \in \E_i$, one for each natural number $i$, such that
$S(X_{i+1}) = T(X_{i+1}) = X_i$ for all $i$.
\end{Defn}
To describe the additional structure required to make a globular object into a
$P$-algebra, we employ one of the central constructions
of~\cite{Batanin1998Monoidal}. This associates to each globular object $X \in
\E$ an endomorphism operad $[X, X]$; which allows us to define a $P$-algebra in
$\E$ to be a globular object $X$ together with a globular operad morphism $P
\to [X, X]$. We now describe the construction of $[X, X]$. First observe that
if $\E$ is a monoidal globular category, then so too is $\E_{+1}$, where
$(\E_{+1})_n = \E_{n+1}$ and the remaining data is defined in the obvious way.
Moreover, if $X$ is a globular object in $\E$, then $X_{+1}$ is a globular
object in $\E_{+1}$, where again we define $(X_{+1})_n = X_{n+1}$. Now, given a
globular object $X \in \E$ and a pasting diagram $\pi \in (T1)_n$, we define by
induction on $\pi$ an object $X^{\otimes \pi} \in \E_n$:
\begin{itemize}
\item If $\pi = \star$, then $X^{\otimes \pi} \defeq X_0 \in \E_0$;
\item If $\pi = (\pi_1, \dots, \pi_k)$, then $X^{\otimes\pi} \defeq
    (X_{+1})^{\otimes \pi_1} \otimes_0 \cdots \otimes_0 (X_{+1})^{\otimes
    \pi_k}$.
\end{itemize}
\begin{Prop}\label{batanin}
Let $\E$ be a monoidal globular category and $X \in \E$ a globular object. Then
there is a globular operad $[X, X]$ with
\begin{equation*}
    [X, X]_\pi \defeq \E_n(X^{\otimes \pi}, X_n) \qquad \text{for all $\pi \in
    (T1)_n$\text.}
\end{equation*}
\end{Prop}
\begin{proof}
This is Proposition~7.2 of~\cite{Batanin1998Monoidal}.
\end{proof}
We now use this result to prove Proposition~\ref{endo}. The first step is to
construct, from our identity type category $\C$, a monoidal globular category
$\E(\C)$.
\begin{Defn}
Let $\mathbb G$ denote the category
\begin{equation*}
    \cd[@C-0.5em]{
    0 & 1 \ar@<3pt>[l]^-{\tau} \ar@<-3pt>[l]_-{\sigma} & 2\ar@<3pt>[l]^-{\tau} \ar@<-3pt>[l]_-{\sigma} & \,\cdots \ar@<3pt>[l]^-{\tau} \ar@<-3pt>[l]_-{\sigma} }
\end{equation*}
The \emph{generic $n$-span} $\mathbb S_n$ is defined to be the coslice category
$n / \mathbb G$. In low dimensions, we have that:
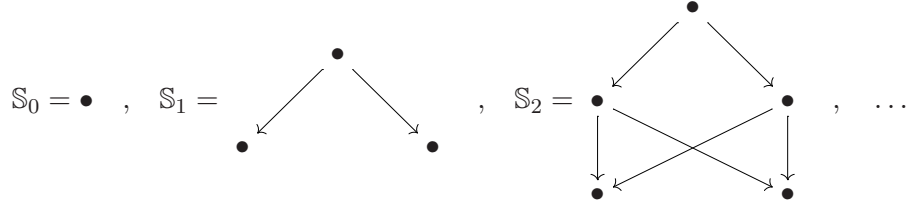
\begin{equation*}
    \mathbb S_0 = \bullet \quad \text, \quad \mathbb S_1 = \cd{ & \bullet \ar[dl] \ar[dr]
    \\ \bullet && \bullet} \quad \text, \quad \mathbb S_2 = \cd{ & \bullet \ar[dl] \ar[dr]
    \\ \bullet \ar[drr] \ar[d] && \bullet \ar[dll] \ar[d] \\
    \bullet && \bullet} \quad \text, \quad \dots
\end{equation*}
The monoidal globular category $\E(\C)$ is defined by taking $\E(\C)_n$ to be
the full subcategory of the functor category $\C^{\mathbb S_n}$ on those
functors which send every morphism of $\mathbb S_n$ to a $\P$-map. The
remaining structure of $\E(\C)$ may be found described in~\cite[Definition
3.2]{Batanin1998Monoidal}. As a representative sample, we describe on objects
the functor $Z \colon \E(\C)_1 \to \E(\C)_2$, which is given by
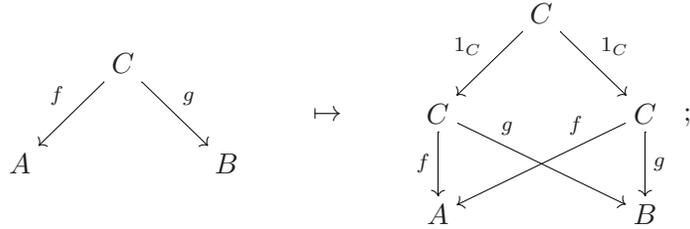
\begin{equation*}
\cd{ & C \ar[dl]_f \ar[dr]^g
    \\ A && B} \qquad \mapsto \qquad \cd{ & C \ar[dl]_{1_C} \ar[dr]^{1_C}
    \\ C \ar[drr]^(0.3)g \ar[d]_f && C \ar[dll]_(0.3)f \ar[d]^g \\
    A && B}\ \text;\end{equation*}
and the functor $\otimes_0 \colon \E(\C)_2 \times_0 \E(\C)_2 \to \E(\C)_2$,
which sends the object
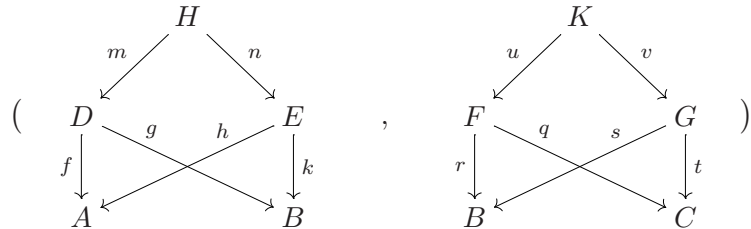
\begin{equation*}
    \big( \quad \cd{ & H \ar[dl]_{m} \ar[dr]^{n}
    \\ D \ar[drr]^(0.3)g \ar[d]_f && E \ar[dll]_(0.3)h \ar[d]^k \\
    A && B} \qquad \text, \qquad
    \cd{ & K \ar[dl]_{u} \ar[dr]^{v}
    \\ F \ar[drr]^(0.3)q \ar[d]_r && G \ar[dll]_(0.3)s \ar[d]^t \\
    B && C} \quad \big)
\end{equation*}
of $\E(\C)_2 \times_0 \E(\C)_2$ to the object
\begin{equation*}
    \cd{ & H \times_B K \ar[dl]_{m \times_B u} \ar[dr]^{n \times_B v}
    \\ D \times_B F \ar[drr]^(0.3){q\pi_2} \ar[d]_{f \pi_1} && E \times_B G \ar[dll]_(0.3){h \pi_1} \ar[d]^{t\pi_2} \\
    A && C}
\end{equation*}
of $\E(\C)_2$. Note that the requisite pullbacks exist by virtue of the
requirement that every arrow in the above diagrams should be a $\P$-map.
%
\end{Defn}
We next observe that, if $\Gamma$ is a globular context in $\C$, then there is
an associated globular object $X_\Gamma \in \E(\C)$ where $(X_\Gamma)_n$ is the
$n$-span
\begin{equation*}
    \cd{ & \Gamma_n \ar[dl]_{s} \ar[dr]^{t}
    \\ \Gamma_{n-1} \ar[drr]^(0.3){t} \ar[d]_{s} && \Gamma_{n-1} \ar[dll]_(0.3){s} \ar[d]^{t} \\
    {\vdots} \ar[drr]^(0.3){t} \ar[d]_{s} & & {\vdots} \ar[dll]_(0.3){s} \ar[d]^{t} \\
    \Gamma_0 && \Gamma_0}\ \text.
\end{equation*}
By a straightforward induction on $\pi$, we may now prove that for any $\pi \in
(T1)_n$, $(X_\Gamma)^{\otimes \pi} \in \E_n$ is given by the $n$-span
\begin{equation*}
    \cd{ & \Gamma^{\pi} \ar[dl]_{\sigma} \ar[dr]^{\tau}
    \\ \Gamma^{\partial \pi} \ar[drr]^(0.3){\tau} \ar[d]_{\sigma} && \Gamma^{\partial \pi} \ar[dll]_(0.3){\sigma} \ar[d]^{\tau} \\
    {\vdots} \ar[drr]^(0.3){\tau} \ar[d]_{\sigma} & & {\vdots} \ar[dll]_(0.3){\sigma} \ar[d]^{\tau} \\
    \Gamma^{\star} && \Gamma^\star}\ \text;
\end{equation*}
from which it follows that the hom-set $\E(\C)_n\big((X_\Gamma)^{\otimes \pi},
(X_\Gamma)_n\big)$ is precisely the set of commutative diagrams of the
form~\eqref{operation}. This allows us to complete the proof of
Proposition~\ref{endo}: indeed, we may take the globular operad $[\Gamma,
\Gamma]$ whose existence is asserted there to be the globular operad
$[X_\Gamma, X_\Gamma]$ whose existence is assured by Proposition~\ref{batanin}.

\begin{Defn}\label{palg}
Let $\C$ be an identity type category. An \emph{internal $P$-algebra} for a
globular operad $P$ is a pair $(\Gamma, f)$, where $\Gamma$ is a globular
context in $\C$ and $f \colon P \to [\Gamma, \Gamma]$ a map of globular
operads. By a \emph{weak $\omega$-category} in $\C$, we mean a triple $(P,
\Gamma, f)$, where $P$ is a normalised, contractible globular operad and
$(\Gamma, f)$ an internal algebra for it.
\end{Defn}
It remains to extend this definition to one of weak $\omega$-groupoid in $\C$.
To do this, we exploit the characterisation of weak $\omega$-groupoids given by
Proposition~\ref{characterisation}.
\begin{Defn}\label{chduals}
Let $f \colon P \to [\Gamma, \Gamma]$ be a weak $\omega$-category in the
identity type category $\C$. Now a \emph{choice of duals} for $\Gamma$, with
respect to some system of compositions $(i_n, m_n)$ on $P$, is given by maps
\begin{align*}
    (\thg)^\ast & \colon \Gamma_n \to \Gamma_n\\
    \eta & \colon \Gamma_n \to \Gamma_{n+1}\\
    \epsilon & \colon \Gamma_n \to \Gamma_{n+1}
\end{align*}
for each $n \geqslant 1$, making the following diagrams commute:
\begin{equation}\label{diags}
\begin{gathered}
    \cd{
      \Gamma_n \ar[rr]^{(\thg)^\ast} \ar[dr]_{(s,t)} & & \Gamma_n \ar[dl]^{(t,s)}
      \\ &
      \Gamma_{n-1} \times \Gamma_{n-1}}
\\    \cd{
      \Gamma_n \ar[r]^{\eta} \ar[d]_{s} & \Gamma_{n+1} \ar[d]^{s}
      \\
      \Gamma_{n-1} \ar[r]_{[i_n]} & \Gamma_n
    } \qquad
        \cd{
      \Gamma_n \ar[r]^{\eta} \ar[d]_{((\thg)^\ast, \id)} & \Gamma_{n+1} \ar[d]^{t}
      \\
      \Gamma_n \mathbin{{}_s\!\times_t} \Gamma_n \ar[r]_-{[m_n]} & \Gamma_n
    }
\\    \cd{
      \Gamma_n \ar[r]^{\epsilon} \ar[d]_{t} & \Gamma_{n+1} \ar[d]^{t}
      \\
      \Gamma_{n-1} \ar[r]_{[i_n]} & \Gamma_n
    } \qquad
        \cd{
      \Gamma_n \ar[r]^{\eta} \ar[d]_{(\id, (\thg)^\ast)} & \Gamma_{n+1} \ar[d]^{s}
      \\
      \Gamma_n \mathbin{{}_s\!\times_t} \Gamma_n \ar[r]_-{[m_n]} & \Gamma_n
    }\ \text.
\end{gathered}
\end{equation}
We say that $(\Gamma, f)$ is a \emph{weak $\omega$-groupoid} if it has a choice
of duals with respect to every system of compositions on $P$.
\end{Defn}

\subsection{Types are weak $\omega$-groupoids}\label{proof2}
We are now ready to prove our main theorem.  It will follow from a general
result that shows a particular class of globular contexts to admit a weak
$\omega$-groupoid structure.
\begin{Defn}
Let $\C$ be an identity type category. A globular context $\Gamma$ is said to
be \emph{reflexive} if it comes equipped with morphisms
\begin{equation*}
    \cd[@C-0.5em]{
    \Gamma_0 \ar@<3pt>@/^10pt/[r]^{r_0} & \Gamma_1 \ar@<3pt>@/^10pt/[r]^{r_1} \ar@<3pt>[l]^-{t} \ar@<-3pt>[l]_-{s} & \Gamma_2 \ar@<3pt>@/^10pt/[r]^{r_2} \ar@<3pt>[l]^-{t} \ar@<-3pt>[l]_-{s} & \,\cdots \ar@<3pt>[l]^-{t} \ar@<-3pt>[l]_-{s}}
\end{equation*}
where each $r_i$ is an $\I$-map satisfying $sr_i = tr_i = \id_{\Gamma_{i}}$.
\end{Defn}
\begin{Thm}\label{reduce}
Every reflexive globular context $(\Gamma, r_i)$ admits a structure of weak
$\omega$-groupoid.
\end{Thm}
To prove the theorem, we first exhibit a weak $\omega$-category structure, and
then show this to be a weak $\omega$-groupoid. To obtain the $\omega$-category
structure, we show the endomorphism operad $[\Gamma, \Gamma]$ of
Proposition~\ref{endo} to admit a normalised, contractible suboperad $P$;
whereupon the inclusion of operads $P \hookrightarrow [\Gamma, \Gamma]$
exhibits $\Gamma$ as a $P$-algebra, and hence a weak $\omega$-category.
\begin{Defn}\label{pointings}
Let $(\Gamma, r_i)$ be a reflexive globular context. We define, for each $\pi
\in (T1)_n$, a map $r_\pi \colon \Gamma_0 \to \Gamma^\pi$ by induction on
$\pi$. If $\pi = \star$, then we take $r_\pi \defeq \id_{\Gamma_0} \colon
\Gamma_0 \to \Gamma_0$. Otherwise, if $\pi = (\pi_1, \dots, \pi_k)$, then we
first observe that $(\Gamma_{+1}, r_{+1})$ is a reflexive globular context in
$\C_{\Gamma_0 \times \Gamma_0}$, where $(r_{+1})_n \defeq r_{n+1}$. Hence by
induction, we obtain, for each $1 \leqslant i \leqslant k$, maps
\begin{equation}\label{dashmaps}
 \cd{    \Gamma_1 \ar[dr]_{(s,t)}
    \ar[rr]^{r'_{\pi_i}} & & (\Gamma_+)^{\pi_i} \ar[dl]^{(s,t)} \\ & \Gamma_0 \times \Gamma_0}
\end{equation}
in $\C_{\Gamma_0 \times \Gamma_0}$. These now give rise to a diagram
\begin{equation}\label{thediag}
    \cd[@-1em]{ & & & \Gamma_0 \ar[dll]_{r'_{\pi_1} \circ r_0} \ar[d] \ar[drr]^{r'_{\pi_k} \circ r_0} \\
        & (\Gamma_{+1})^{\pi_1} \ar[dl]_s \ar[dr]^t & & \ \ \ \cdots \ \ \ \ar[dl]_s \ar[dr]^t & & (\Gamma_{+1})^{\pi_k} \ar[dl]_s
        \ar[dr]^t \\
        \Gamma_0 & & \Gamma_0 & & \Gamma_0 & & \Gamma_0
    }\ \text,
\end{equation}
wherein, by a straightforward calculation, any map from $\Gamma_0$ at the top
to some $\Gamma_0$ at the bottom is an identity. In particular, this means that
$\Gamma_0$, together with the maps out of it, form a cone over the remainder of
the diagram. But $\Gamma^\pi$ is, by definition, the limit of this subdiagram,
and so we induce a map $r_\pi \colon \Gamma_0 \to \Gamma^\pi$ as required.
\end{Defn}
\begin{Prop}\label{endo2} Let $(\Gamma, r_i)$ be a reflexive globular
context in $\C$. Then the globular operad $[\Gamma, \Gamma]$ has a suboperad
$P$ whose set  of operations of shape $\pi$ comprises all serially commutative
diagrams of the form~\eqref{operation} in which the $f_i$'s and $g_i$'s commute
with the pointings $r_\pi \colon \Gamma_0 \to \Gamma^\pi$ of
Definition~\ref{pointings}.
\end{Prop}
\begin{proof}
Let us write $\Gamma_\ast$ to denote the globular context
\begin{equation*}
    \cd[@C+0.5em]{ & \Gamma_0 \ar[dl]_{\id} \ar[d]|{r_0}
    \ar[dr]|{r_0r_1} \\
    \Gamma_0 & \Gamma_1 \ar@<3pt>[l]^-{t} \ar@<-3pt>[l]_-{s} & \Gamma_2 \ar@<3pt>[l]^-{t} \ar@<-3pt>[l]_-{s} & \,\cdots \ar@<3pt>[l]^-{t}
    \ar@<-3pt>[l]_-{s}}
\end{equation*}
in the identity type category $\Gamma_0 / \C$. We claim that that the object
$(\Gamma_\ast)^\pi \in \Gamma_0 / \C$ is given by $r_\pi \colon \Gamma_0 \to
\Gamma^\pi$. Observe that this implies the result, because the endomorphism
operad $[\Gamma_\ast, \Gamma_\ast]$ is then precisely the suboperad $P \subset
[\Gamma, \Gamma]$ we require. We will prove the claim by induction on $\pi$.
When $\pi = \star$ it is clear. So suppose now that $\pi = (\pi_1, \dots,
\pi_k)$. By the description given in Definition~\ref{idxd}, and the inductive
hypothesis, we see that $(\Gamma_\ast)^\pi$ is given by the unique map
$\Gamma_0 \to \Gamma^\pi$ induced by the following cone:
\begin{equation*}
    \cd[@-1em]{ & & & \Gamma_0 \ar[dll]_{r_{\pi_1}} \ar[d] \ar[drr]^{r_{\pi_k}} \\
        & (\Gamma_{+1})^{\pi_1} \ar[dl]_s \ar[dr]^t & & \ \ \ \cdots \ \ \ \ar[dl]_s \ar[dr]^t & & (\Gamma_{+1})^{\pi_k} \ar[dl]_s
        \ar[dr]^t \\
        \Gamma_0 & & \Gamma_0 & & \Gamma_0 & & \Gamma_0
    }\ \text.
\end{equation*}
Thus, it suffices to show that this cone coincides with~\eqref{thediag}; which
is to show that, for each $1 \leqslant i \leqslant k$, we have $r_{\pi_i} =
r'_{\pi_i} \circ r_0$. Now, observe that $r_{\pi_i}$ is obtained as
$((\Gamma_\ast)_{+1})^{\pi_i}$, where $(\Gamma_\ast)_{+1}$ is the globular
context
\begin{equation*}
    \cd[@C+0.5em]{ & \Gamma_0 \ar[dl]_{r_0} \ar[d]|{r_0r_1}
    \ar[dr]^{r_0r_1r_2} \\
    \Gamma_1 \ar[dr] & \Gamma_2 \ar@<3pt>[l]^-{t} \ar[d] \ar@<-3pt>[l]_-{s} & \Gamma_3 \ar[dl] \ar@<3pt>[l]^-{t} \ar@<-3pt>[l]_-{s} & \,\cdots \ar@<3pt>[l]^-{t}
    \ar@<-3pt>[l]_-{s} \\ & \Gamma_0 \times \Gamma_0}
\end{equation*}
in $\Gamma_0 / \C_{\Gamma_0 \times \Gamma_0}$. On the other hand, by a further
application of the inductive hypothesis, $r'_{\pi_i}$ is obtained as the map
$((\Gamma_{+1})_\ast)^{\pi_i}$, where $(\Gamma_{+1})_\ast$ is the globular
context
\begin{equation*}
    \cd[@C+0.5em]{ & \Gamma_1 \ar[dl]_{\id} \ar[d]|{r_1}
    \ar[dr]^{r_1r_2} \\
    \Gamma_1 \ar[dr] & \Gamma_2 \ar[d] \ar@<3pt>[l]^-{t} \ar@<-3pt>[l]_-{s} & \Gamma_3 \ar[dl] \ar@<3pt>[l]^-{t} \ar@<-3pt>[l]_-{s} & \,\cdots \ar@<3pt>[l]^-{t}
    \ar@<-3pt>[l]_-{s} \\ & \Gamma_0 \times \Gamma_0}
\end{equation*}
in $\Gamma_1 / \C_{\Gamma_0 \times \Gamma_0}$. But the functor $(r_0)_! \colon
\Gamma_1 / \C_{\Gamma_0 \times \Gamma_0} \to \Gamma_0 / \C_{\Gamma_0 \times
\Gamma_0}$ given by precomposition with the map $r_0 \colon \Gamma_0 \to
\Gamma_1$ of $\C_{\Gamma_0 \times \Gamma_0}$ sends the latter of these globular
contexts to the former; and thus, because $(r_0)_!$ preserves limits, it must
also send $((\Gamma_{+1})_\ast)^{\pi_i}$ to $((\Gamma_\ast)_{+1})^{\pi_i}$:
which is to say that $r_{\pi_i} = r'_{\pi_i} \circ r_0$ as required.
\end{proof}
Thus, for a reflexive globular context $(\Gamma, r_i)$, we have now defined the
suboperad $P \subset [\Gamma, \Gamma]$ required for the proof of
Theorem~\ref{reduce}. It remains only to show that $P$ is normalised and
contractible. To do this, we will need:
\begin{Prop}\label{allimaps}
Let $(\Gamma, r_i)$ be a reflexive globular context in $\C$. Then each of the maps
$r_\pi \colon \Gamma_0 \to \Gamma^\pi$ of Definition~\ref{pointings} is an
$\I$-map.
\end{Prop}
\begin{proof}
We proceed by induction on $\pi$. When $\pi = \star$, we have $r_\pi$ an
identity map, and hence an $\I$-map. So suppose now that $\pi = (\pi_1, \dots,
\pi_k)$, and consider the diagram~\eqref{thediag} defining the map $r_\pi \colon
\Gamma_0 \to \Gamma^\pi$. In it, each of the maps $r'_{\pi_i}$ is an $\I$-map
by induction, and so because $r_0$ is an $\I$-map by assumption, and $\I$-maps
are closed under composition, $r'_{\pi_i} \circ r_0$ is also an $\I$-map.
Repeated application of the following lemma now completes the proof.
\end{proof}
\begin{Lemma}
Suppose that
\begin{equation*}
    \cd[@-1em]{& A \ar[dl]_i \ar[dr]^j \ar[dd]|{\id_A} \\ B \ar[dr]_p & & C \ar[dl]^q \\ &
    A}
\end{equation*}
is a commutative diagram in an identity type category $\C$. Suppose further
that $i$ and $j$ are $\I$-maps, and $p$ and $q$ are $\P$-maps. Then the induced
map $(i, j) \colon A \to B \times_A C$ is also an $\I$-map.
\end{Lemma}
\begin{proof}
We first form the pullback square
\begin{equation*}
    \cd{B \times_A C \ar[d]_{p'} \ar[r]^-{q'} & B \ar[d]^p \\
    C \ar[r]_q & A} \ \text.
\end{equation*}
Now the universal property of this pullback induces a factorisation of the
commutative square
\begin{equation*}
    \cd{B \ar[d]_{jp} \ar[r]^{\id_B} & B \ar[d]^p \\
    C \ar[r]_q & A}
\end{equation*}
as
\begin{equation*}
    \cd{B \ar[d]_{p} \ar[r]^-{j'} &
    B \times_A C \ar[d]^{p'} \ar[r]^-{q'} &
    B \ar[d]^p \\
    A \ar[r]_j
     & C \ar[r]_q
     & A} \ \text.
\end{equation*}
Since the outer rectangle has identities along both horizontal edges, it is a
pullback. But the right-hand square is a pullback, and so we deduce that the
left-hand square is too. Now $p'$ is a $\P$-map by (Stability), and $j$ is an
$\I$-map by assumption, and so by (Frobenius), $j'$ is also an $\I$-map. It
follows, by (Composition) and the fact that $i$ is an $\I$-map, that
\begin{equation*}
    A \xrightarrow i B \xrightarrow{j'} B \times_A C
\end{equation*}
is also an $\I$-map. But this map is the induced map $(i, j) \colon A \to B
\times_A C$, since it has $i$ as its projection onto $B$, and $jpi = j$ as its
projection onto $C$.
\end{proof}
\begin{Prop}
Let $(\Gamma, r_i)$ be a reflexive globular context in $\C$. Then the suboperad
$P \subset [\Gamma, \Gamma]$ of Proposition~\ref{endo2} is both normalised and
contractible.
\end{Prop}
\begin{proof}
Note first that the set $P_\star$ is the set of all morphisms $f_0 \colon
\Gamma_0 \to \Gamma_0$ for which $f_0 \circ \id_{\Gamma_0} = \id_{\Gamma_0}$
and hence a singleton, which proves $P$ is normalised. To show it contractible,
we must show that, given a serially commutative diagram of the form
\begin{equation}\label{tocompl2}
   \cd{
     \Gamma^\pi \ar@<3pt>[d]^{\tau} \ar@<-3pt>[d]_{\sigma} &
     \Gamma_n \ar@<3pt>[d]^{t} \ar@<-3pt>[d]_{s} \\
     \Gamma^{\partial\pi} \ar@<3pt>[r]^{f_{n-1}} \ar@<-3pt>[r]_{g_{n-1}} \ar@<3pt>[d]^{\tau} \ar@<-3pt>[d]_{\sigma} &
     \Gamma_{n-1} \ar@<3pt>[d]^{t} \ar@<-3pt>[d]_{s} \\
     \vdots \ar@<3pt>[d]^{\tau} \ar@<-3pt>[d]_{\sigma} & \vdots \ar@<3pt>[d]^{t}
     \ar@<-3pt>[d]_{s} \\
     \Gamma_0 \ar@<3pt>[r]^{f_{0}} \ar@<-3pt>[r]_{g_{0}} & \Gamma_0
   }
\end{equation}
wherein each $f_i$ and $g_i$ commutes with the pointings, we can find a map
$f_n \colon \Gamma^\pi \to \Gamma_n$ completing the diagram (and commuting with
the pointings). First we note that the diagram
\begin{equation*}
\cd[@C+1.5em]{
    \Gamma^{\pi} \ar[r]^-{f_{n-1} \circ \sigma} \ar[d]_{g_{n-1} \circ \tau} & \Gamma_{n-1}
    \ar[d]^{(s,t)} \\
    \Gamma_{n-1} \ar[r]_-{(s,t)} & B_{n-1}\Gamma
}
\end{equation*}
commutes, as may be seen by postcomposing it with the two projections $B_{n-1}
\Gamma \rightrightarrows \Gamma_{n-2}$, and observing the resultant diagrams
commutative. Thus we induce a map $k \colon \Gamma^{\pi} \to B_n\Gamma$. We now
consider the diagram
\begin{equation*}
\cd[@C+3em]{
  \Gamma_0 \ar[d]_{r_\pi} \ar[r]^-{r_{n-1} \cdots r_0} & \Gamma_n \ar[d]^{(s,t)} \\
  \Gamma^\pi \ar[r]_{k} & B_n\Gamma
} \ \text.
\end{equation*}
That this is commutative once again follows from the fact that it is so upon
postcomposition with the two projections $B_{n} \Gamma \rightrightarrows
\Gamma_{n-1}$. Moreover, $r_\pi$ is an $\I$-map by Proposition~\ref{allimaps},
and $(s,t)$ a $\P$-map by the definition of globular context; so that by
(Orthogonality), we can find a map $f_n \colon \Gamma^\pi \to \Gamma_n$ making
both induced triangles commute. That the lower triangle commutes says that
$f_n$ renders the diagram~\eqref{tocompl2} serially commutative; whilst that
the upper triangle commutes says that $f_n$ commutes with the pointings.
\end{proof}

Thus we have shown the operad $P \subset [\Gamma, \Gamma]$ to be normalised and
contractible, from which it follows that the inclusion $P \hookrightarrow
[\Gamma, \Gamma]$ exhibits $\Gamma$ as a weak $\omega$-category. It remains to
show that this weak $\omega$-category is a weak $\omega$-groupoid.

\begin{Prop}
Let $(\Gamma, r_i)$ be a reflexive globular context in $\C$, and let $P \subset
[\Gamma, \Gamma]$ be the operad defined above. Then the inclusion $P
\hookrightarrow [\Gamma, \Gamma]$ exhibits $\Gamma$ as a weak
$\omega$-groupoid.
\end{Prop}
\begin{proof}
According to Definition~\ref{chduals}, we must show that for any given system
of compositions $(i_n, m_n)$ for $P$, there is a corresponding choice of duals
for $\Gamma$. Now, for each $n \geqslant 1$ we have a commutative diagram
\begin{equation*}
\cd[@C+3em]{
  \Gamma_0 \ar[d]_{r_{n-1} \cdots r_0} \ar[r]^-{r_{n-1} \cdots r_0} & \Gamma_n \ar[d]^{(t,s)} \\
  \Gamma_n \ar[r]_{(s,t)} & B_n\Gamma
}\ \text.
\end{equation*}
The left-hand morphism is an $\I$-map, and the right-hand one a $\P$-map; and
so by (Orthogonality) we have a diagonal filler $(\thg)^\ast \colon \Gamma_n
\to \Gamma_n$. Commutativity of the lower triangle implies the commutativity of
the first diagram in~\eqref{diags}. We induce $\eta$ and $\epsilon$ similarly,
by considering the commutative squares
\begin{equation*}
\cd[@C+7em]{
  \Gamma_0 \ar[d]_{r_{n-1} \cdots r_0} \ar[r]^-{r_n r_{n-1} \cdots r_0} & \Gamma_{n+1} \ar[d]^{(s,t)} \\
  \Gamma_n \ar[r]_{\big([i_n]s\c [m_n] \circ ((\thg)^\ast, \id)\big)} & B_{n+1}\Gamma
}
\end{equation*}
\begin{equation*}
\cd[@C+7em]{
  \Gamma_0 \ar[d]_{r_{n-1} \cdots r_0} \ar[r]^-{r_n r_{n-1} \cdots r_0} & \Gamma_{n+1} \ar[d]^{(s,t)} \\
  \Gamma_n \ar[r]_{\big([m_n] \circ (\id, (\thg)^\ast)\c [i_n]t\big)} & B_{n+1}\Gamma
}\ \text.
\end{equation*}
Again, commutativity of the lower triangles entails the commutativity of the
remaining four diagrams in~\eqref{diags}.
\end{proof}

We have thus shown that every reflexive globular context in an identity type
category $\C$ bears a structure of weak $\omega$-groupoid. Note that in giving
this proof, we have nowhere used the axiom (Identities). In fact, the only
reason we need it is to show that from an object of $\C$ we can construct a
reflexive globular context corresponding to its tower of identity types.
\begin{Defn}\label{towerid}
Let $\C$ be an identity type category and $A \in \C$. We define a reflexive
globular context $\underline A \in \C$ by the following induction. For the base
case, we take $A_0 = A$. For the inductive step, suppose we have defined $A_0,
\dots, A_n$. Then we may form the \mbox{$n$-dim\-ensional} boundary $B_n
\underline A$ of $A$, and by induction the map $(s,t) \colon A_n \to B_n
\underline A$ is a $\P$-map. So by (Identities), we may factorise the diagonal
morphism $A_n \to A_n \times_{B_n \underline A} A_n$ as
\begin{equation*}
\cd{    A_n \ar[r]^-{r_{n+1}} & \Id(A_n) \ar[r]^-{e_{n+1}} & A_n \times_{B_n
\underline A} A_n}\ \text,
\end{equation*}
with $r_{n+1}$ an $\I$-map and $e_{n+1}$ a $\P$-map. We now define $A_{n+1}$ to
be $\Id(A_n)$, and $s, t \colon A_{n+1} \to A_n$ to be the composites of
$e_{n+1}$ with the two projection morphisms $A_n \times_{B_n \underline A} A_n
\to A_n$. It remains to show that the induced map $(s, t) \colon A_{n+1} \to
B_{n+1} \underline A$ is a $\P$-map. But we recall that $B_{n+1} \underline A$
was defined by the pullback diagram~\eqref{boundary}, so that $B_{n+1}
\underline A = A_n \times_{B_n \underline A} A_n$, and the induced map $(s, t)$
is precisely $e_{n+1}$, which is, by assumption, a $\P$-map.
\end{Defn}

This definition, together with Theorem~\ref{reduce} now immediately imply:

\begin{Thm}\label{main}
Let $\C$ be an identity type category and $A \in \C$. Then the globular context
$\underline A$ is a weak $\omega$-groupoid in $\C$.
\end{Thm}

In particular, taking $\C$ to be the identity type category $\C_\mathbb T$
associated with some dependent type theory $\mathbb T$, we obtain:
\begin{Thm}
Let $\mathbb T$ be a dependent type theory admitting each of the rules
described in Section~\ref{mltt}. Then for each type $A$ of $\mathbb T$, the
tower of identity types over $A$ is a weak $\omega$-groupoid.
\end{Thm}


\end{document}